\documentclass[reqno,10pt,centertags]{amsart} 
\usepackage{amsmath,amsthm,amscd,amssymb,latexsym,esint,upref,stmaryrd,
	enumerate,color,verbatim,yfonts,enumitem}
\usepackage{hyperref} 
\usepackage{esint}
\usepackage{mathrsfs}
\newcommand*{\mailto}[1]{\href{mailto:#1}{\nolinkurl{#1}}}

\setenumerate{label=$($\textit{\roman*}$)$}

\newcommand{\C}{{\mathbb C}}

\newcommand{\bbR}{{\mathbb{R}}}

\DeclareMathOperator{\ran}{ran}
\DeclareMathOperator{\dom}{dom}

\renewcommand{\Re}{\text{\rm Re}}
\renewcommand{\Im}{\text{\rm Im}}
\renewcommand{\ln}{\text{\rm ln}}

\renewcommand{\dot}{\overset{\textbf{\Large.}}}

\newcommand{\bi}{\bibitem}

\makeatletter
\def\theequation{\@arabic\c@equation}

\allowdisplaybreaks 
\numberwithin{equation}{section}
\usepackage{tikz}
\newtheorem{theorem}{Theorem}[section]
\newtheorem{proposition}[theorem]{Proposition}
\newtheorem{lemma}[theorem]{Lemma}
\newtheorem{corollary}[theorem]{Corollary}
\newtheorem{definition}[theorem]{Definition}
\newtheorem{hypothesis}[theorem]{Hypothesis}

\theoremstyle{remark}

\newenvironment{remark}[1][]{\refstepcounter{theorem}\par\medskip\noindent\textit{Remark~$\theexample. #1$} \rmfamily}{{\ }\hfill $\diamond$ \vspace{6pt}}

\begin{document}
	
	\title[Non-local point interactions]{An analysis of non-selfadjoint first-order differential operators with non-local point interactions}
	
	\author[C.\ Fischbacher]{Christoph Fischbacher} 
	\address[C. Fischbacher]{Department of Mathematics, 
		Baylor University, Sid Richardson Bldg, 1410 S.\,4th Street, Waco, TX 76706, USA}
	\email{\mailto{C\_Fischbacher@baylor.edu}}
	\urladdr{\url{https://math.artsandsciences.baylor.edu/person/christoph-fischbacher-phd}}

	\author[D.\ Paraiso]{Danie Paraiso}
	\address[D.\ Paraiso]{}
	\email{\mailto{Danie\_Paraiso1@bayor.edu}}
	
	\author[C.\ Povey-Rowe]{Chloe Povey-Rowe}
	\address[C.\ Povey-Rowe]{}
	\email{\mailto{Chloe\_Poveyrowe1@baylor.edu}}
	
	\author[B.\ Zimmerman]{Brady Zimmerman}
	\address[B.\ Zimmerman]{}
	\email{\mailto{Brady\_Zimmerman1@baylor.edu}}

	\date{\today}
	\@namedef{subjclassname@2020}{\textup{2020} Mathematics Subject Classification}
	
	\begin{abstract} We study the spectra of non-selfadjoint first-order operators on the interval with non-local point interactions, formally given by ${i\partial_x+V+k\langle \delta,\cdot\rangle}$. We give precise estimates on the location of the eigenvalues on the complex plane and prove that the root vectors of these operators form Riesz bases of $L^2(0,2\pi)$. Under the additional assumption that the operator is maximally dissipative, we prove that it can have at most one real eigenvalue, and given any $\lambda\in\mathbb{R}$, we explicitly construct the unique operator realization such that $\lambda$ is in its spectrum. We also investigate the time-evolution generated by these maximally dissipative operators.
		
	\end{abstract}
	
	\maketitle
	

	\section{Introduction}\label{sec:1}
	The purpose of this paper is to study the spectra and root functions of first-order differential operators of the form
	\begin{equation} \label{eq:1.1}
		A: f\mapsto if'+Vf+f(2\pi)k\:,
	\end{equation}
	on the interval $(0,2\pi)$, where $V$ is a complex-valued potential. The additional term $``+f(2\pi)k"$ is referred to as a ``non-local point interaction" (cf.\ \cite{AN07}) formally corresponding to a singular rank-one perturbation of the form $``k\langle \delta_{2\pi},\cdot\rangle"$, where $\delta_{2\pi}$ is a Dirac point interaction at the right endpoint $2\pi$. There has been a significant amount of research dedicated to operators with non-local point interactions in recent years \cite{AHN07, AN07, AN13, CN19, DN23, DN19, N09, N10, N11, N12} ranging from first-order differential operators to Dirac and Schr\"odinger operators. However, to the best of our knowledge, all these realizations of differential operators with non-local point interactions were studied in the selfadjoint context, where in addition to the non-local point interaction of the form $``k\langle \delta,\cdot\rangle"$ occurring in the action of the operator, in order to ensure selfadjointness, one has to add another non-local term of the form $``\delta\langle k,\cdot\rangle"$, which corresponds to a non-local boundary condition occurring in the domain of the operator. By studying operators of the form \eqref{eq:1.1} while imposing local boundary conditions of the form $f(0)=\rho f(2\pi)$ for some $\rho\in\mathbb{C}$, we go beyond the selfadjoint setup and make a first contribution to the non-selfadjoint setting. 
	
	The study of these particular operators has a twofold motivation: $(i)$ they arise very naturally when considering maximally dissipative extensions of typical minimal realizations of dissipative differential operators \cite{F17, F18, F19, F21, FNW16, FNW24} and $(ii)$ the first-order case is more accessible to direct calculation thus laying foundational work for future investigations of higher-order non-selfadjoint differential operators with non-local point interactions. 
	
	The two main theorems of this paper are Theorem \ref{thm:3.14}, where we give a very precise estimate of the location of the eigenvalues and Theorem \ref{thm:3.20}, where we prove that if $\rho\neq 0$, the root functions of $A_{\rho,k}$ and $A_{\rho,k}^*$ form a Riesz basis of $L^2(0,2\pi)$ (see also Corollary \ref{thm:4.5}).
	
	We will proceed as follows:
	
	In Section \ref{sec:2}, we introduce some notation, recall basic definitions and results necessary for the results presented later. 
	
	In Section \ref{sec:3}, we introduce the minimal operator $A_{min}$ and its extensions $A_{\rho,k}$, where the scalar $\rho\in\mathbb{C}$ is a parameter describing the boundary condition and $k\in L^2(0,2\pi)$ is a Hilbert space-valued parameter describing the non-local point interaction. Since the operator under consideration is first-order, we derive an explicit formula for the resolvent of $A_{\rho,k}$ (Theorem \ref{thm:3.1}) and derive an entire function $\Phi$ (cf.\ Equation \eqref{eq:3.8}) whose zeros correspond to the eigenvalues of $A_{\rho,k}$. We also show that the order of these zeros is equal to the algebraic multiplicity of the corresponding eigenvalue (Theorem \ref{thm:3.6}). From the explicit formula for the resolvent, one can immediately tell that it is Hilbert-Schmidt (Corollary \ref{coro:3.3}). It is also noteworthy that for different choices of the parameters, the corresponding resolvents differ by at most a rank-one operator. Unlike in the selfadjoint setup, it is possible to go from the resolvent of an operator whose root functions form a Riesz basis to the resolvent of an operator with empty spectrum just by adding a rank-one perturbation (cf.\ Remark \ref{rem:3.2}), which shows that standard perturbation theory will not be applicable to the operators under consideration. 
	In Lemma \ref{lemma:3.4}, we then apply a similarity transformation (integrating factor) which transforms the operators $A_{\rho,k}$, where $\rho\in\mathbb{C}\setminus\{0\}$, to the much simpler operator $(P_{1,K}+\eta)$, where $\eta\in\mathbb{C}$, $\dom(P_{1,K})=\{f\in H^1(0,2\pi)\:|\:f(0)=f(2\pi)\}$ and $(P_{1,K}f)(x)=if'(x)+f(2\pi)K(x)$. Restricting ourselves to this particular operator, we now investigate the zeros of the corresponding entire function $\Phi$ (Theorems \ref{thm:3.7}, \ref{thm:3.11a}, and \ref{thm:3.13}). For our first result, we take advantage of the fact that the zeros of similar functions have been studied in different context by Anselone/Boas and Cartwright in the 1960s \cite{AB62, C64} and use it to show that all but finitely many eigenvalues of $P_{1,K}$ have algebraic multiplicity equal to one and can be found inside small disks around integers. We provide a refinement of this result in the latter two theorems. In Theorem \ref{thm:3.11a}, we show that the radii of these disks are not larger than the absolute value of the Fourier coefficients of $K$ and in Theorem \ref{thm:3.13}, we provide a quantitative statement about the location of the finitely many eigenvalues that may not be found inside small disks around the integers. We believe these results to be of additional independent interest since it adds to the work of Anselone/Boas and Cartwright.
	
	This analysis results in the first of two main results of this paper, Theorem \ref{thm:3.14} summarizing the above results.
	For the second main result of this paper, Theorem \ref{thm:3.20}, we then show that the root functions of $P_{1,K}$ and $P_{1,K}^*$ are quadratically close to the standard orthonormal Fourier basis, thus implying that they form Riesz bases of $L^2(0,2\pi)$.
	
	In Section \ref{sec:4}, we focus on the special situation when $A_{\rho,k}$ is maximally dissipative. We begin by giving necessary and sufficient conditions on $\rho$ and $k$ to be maximally dissipative (Proposition \ref{thm:4.2}). We then address the question of whether it is possible for these operators to have real eigenvalues. We prove in Theorem \ref{thm:5.2} that the spectrum of $A_{\rho,k}$ may contain at most one real eigenvalue and moreover, for every $\lambda\in\mathbb{R}$, we show by an explicit construction that there is a unique specific choice $\rho_\lambda$ and $k_\lambda$ such that $\lambda$ is in the spectrum of $A_{\rho_\lambda,k_\lambda}$ (cf.\ Theorem \ref{thm:4.9}). We finish by showing that if $A_{\rho,k}$ has no real eigenvalue, then the semigroup $e^{itA_{\rho,k}}$ generated by $iA_{\rho,k}$ converges to zero in norm as $t\rightarrow\infty$, and if $A_{\rho,k}$ has a real eigenvalue $\lambda$, then $e^{i(A_{\rho,k}-\lambda)t}$ converges in norm to the projection onto the corresponding eigenspace (Theorem \ref{thm:4.10}).
	\section{Notation and some preliminaries} \label{sec:2}
	Let $A$ be a closed and densely defined linear operator in a separable Hilbert space $\mathcal{H}$ with inner product $\langle \cdot,\cdot\rangle$, which we assume to be linear in the second component and antilinear in the first. We denote its domain, range, and kernel by $\dom(A), \ran(A)$, and $\ker(A)$, respectively. The spectrum of $A$ is denoted $\sigma(A)$, the point spectrum of $A$ by $\sigma_p(A)$, and its resolvent set by $\varrho(A)$. Since we are dealing with non-selfadjoint operators, we need to distinguish between the \emph{geometric multiplicity} of an eigenvalue $\lambda\in\sigma_p(A)$, given by $\dim\ker(A-\lambda)$, and its \emph{algebraic multiplicity} $m_a(\lambda,A)$. To define the latter, we first introduce the \emph{root space} of $A$ corresponding to $\lambda$. Denoted by $\mathcal{R}(\lambda,A)$, it is given by 
	\begin{equation}
		\mathcal{R}(\lambda,A)=\{f\in\mathcal{H}\:|\: \exists k\in\mathbb{N} \mbox{ such that }(A-\lambda)^kf=0\}\:.
	\end{equation}
	Its elements are referred to as \emph{root vectors} or \emph{root functions}. Then, the algebraic multiplicity  $m_a(\lambda,A)$ is given by
	\begin{equation}
		m_a(\lambda,A):=\dim\mathcal{R}(\lambda,A)\:.
	\end{equation}
	Next, we recall the concept of a \emph{Riesz basis}:
	\begin{definition} A system of vectors $\{f_n\}_{n\in\mathbb{Z}}\subseteq\mathcal{H}$ is called a \emph{Riesz basis} if there exists a bounded and boundedly invertible operator $S:\mathcal{H}\rightarrow\mathcal{H}$ and an orthonormal basis $\{\psi_n\}_{n\in\mathbb{Z}}$ of $\mathcal{H}$ such that $f_n=S\psi_n$ for each $n\in\mathbb{Z}$.
	\end{definition}
	\begin{remark}
		It is well-known, cf.\ e.\ g.\ {\cite[\S 2, Chapter VI, Result 1 on page 309]{GK69}}, that every Riesz basis $\{f_n\}_{n\in\mathbb{Z}}$ is in particular a \emph{Schauder basis} of $\mathcal{H}$, which means that for every $g\in\mathcal{H}$, there is a unique expansion of $g$ of the form
		\begin{equation}
			g=\sum_{n\in\mathbb{Z}}c_nf_n\:.
		\end{equation}
	\end{remark}
	
	\section{First-order differential operators with non-local point-interactions} \label{sec:3}
	From now on, let $\mathcal{H}=L^2(0,2\pi)$. The minimal operator $A_{min}$ is given by
	\begin{align}
		A_{min}:\quad\dom(A_{min})&=\{f\in H^1(0,2\pi)\:|\: f(0)=f(2\pi)=0\}\\ (A_{min}f)(x)&= if'(x)+V(x)f(x)\:,
	\end{align}
	where $V\in L^\infty(0,2\pi)$. We decompose $V$ into its real part $V_R$ and its imaginary part $V_I$, i.e.\ $V=V_R+iV_I$. By a slight abuse of notation, we also use the same symbols $V, V_R$, and $V_I$ to denote the associated maximal operators of multiplication by these functions (the latter two being bounded selfadjoint operators in $\mathcal{H}$). 
	Also, let us introduce the maximal operator $A_{max}=A_{min}^*$, which is given by
	\begin{align}
		A_{max}:\quad\dom(A_{max})&=H^1(0,2\pi)\:,\quad (A_{max}f)(x)&= if'(x)+\overline{V(x)}f(x)\:.
	\end{align}
	
	The purpose of this paper is a spectral analysis of extensions of $A_{min}$ that are of the form
	\begin{align}
		A_{\rho,k}: \quad\dom(A_{\rho,k})&=\{f\in H^1(0,2\pi)\:|\: f(0)=\rho f(2\pi)\} \\(A_{\rho,k}f)(x)&=if'(x)+V(x)f(x)+f(2\pi)k(x)\:,
	\end{align}
	where $\rho\in\mathbb{C}\setminus\{0\}$.
	Note that while the operators $A_{\rho,k}$ are extensions of $A_{\min}$, they preserve the differential expression $``i\partial_x+V"$ only if $k(x)=0$ almost everywhere. The additional term $``+f(2\pi)k"$ is referred to as non-local point interaction at $2\pi$. 
	
	Next, let us also determine the adjoint $A_{\rho,k}^*$
	of $A_{\rho,k}$. Since $A_{min}\subseteq A_{\rho,k}$, this immediately implies that $A_{\rho,k}^*\subseteq A_{max}$, so the effect of the non-local point interaction will occur as a non-local boundary condition in the domain of $A_{\rho,k}^*$ while its action will be of the form $``i\partial_x+\overline{V}"$. Indeed, a calculation shows that
	\begin{align} \label{eq:3.6}
		A_{\rho,k}^*:\quad\dom( A_{\rho,k}^*)&=\{f\in H^1(0,2\pi)\:| f(2\pi)=\overline{\rho}f(0)+\langle ik,f\rangle\:\}\\ f&\mapsto if'+\overline{V}f\:.
	\end{align}
	\medskip
	
	\subsection{The spectrum of $A_{\rho,k}$}
	
	We will now give a characterization of the eigenvalues of $A_{\rho,k}$. As we will see, they are solutions to the following transcendental equation, which we will refer to as the ``eigenvalue equation":
	\begin{equation}\label{eq:3.4}
		e^{i \int_{0}^{2\pi}(V(t) - \lambda)\, \mathrm{d}t} \left(i \int_{0}^{2\pi} e^{-i\int_{0}^{t}(V(s) - \lambda) \, \mathrm{d}s}k(t) \, \mathrm{d}t + \rho \right) = 1\:.
	\end{equation}
	
	\begin{theorem} \label{thm:3.1}
		
		The spectrum of $A_{\rho,k}$ consists only of eigenvalues and is given by
		\begin{equation}
			\sigma(A_{\rho,k})=\{\lambda\in \mathbb{C} \:|\: \lambda \mbox{ satisfies \eqref{eq:3.4}}\}\:,
		\end{equation}
		where the geometric multiplicity of each eigenvalue is $\dim(\ker(A_{\rho,k}-\lambda))=1$. The corresponding eigenfunctions are given by
		\begin{equation}  \label{eq:3.10}
			\phi_{\lambda}(x)=\frac{1}{\sqrt{2\pi}}e^{i\int_0^x V(t)dt-i\lambda x}\left(i\int_0^x e^{-i\int_0^t V(s)ds+i\lambda t}k(t)dt+\rho\right)\:.
		\end{equation}
		Moreover, for $\lambda\in \varrho(A_{\rho,k})=\mathbb{C}\setminus\sigma(A_{\rho,k})$, the resolvent operator $(A_{\rho,k}-\lambda)^{-1}$ is given by
		\begin{align}
			&\left((A_{\rho,k}-\lambda)^{-1}g\right)(x)\\=&\frac{-i}{I_{V,\lambda}(x)}\left[\frac{\int_0^{2\pi}I_{V,\lambda}(t)g(t)dt}{I_{V,\lambda}(2\pi)-i\int_0^{2\pi}I_{V,\lambda}(t)k(t)dt-\rho}\left(\int_0^x iI_{V,\lambda}(t)k(t)dt+\rho\right)+\int_0^x I_{V,\lambda}(t)g(t)dt\right]\:, \label{eq:3.27}
		\end{align}
		where 
		\begin{equation} \label{eq:3.15a}
			I_{V,\lambda}(x):=\exp\left(-i\int_0^x V(t)dt+i\lambda x\right)\:. 
		\end{equation}
		
	\end{theorem}
	\begin{proof}
		Consider the eigenvalue equation:
		\begin{equation} \label{eq:3.8}
			A_{\rho,k}f=if'+Vf+f(2\pi)k=\lambda f
		\end{equation}
		subject to the boundary condition $f(0)=\rho f(2\pi)$. We now argue that such $f$ must satisfy $f(2\pi)\neq 0$, since otherwise
		\begin{equation}
			if'+(V-\lambda) f=0
		\end{equation}
		subject to the initial condition $f(0)=f(2\pi)=0$ and consequently $f(x)\equiv 0$, which is a contradiction.
		Hence, without loss of generality, we may assume that $f(2\pi)=1$. Thus, we compute the solution to \eqref{eq:3.8} satisfying $f(2\pi)=1$ and $f(0)=\rho$. In this case, \eqref{eq:3.8} is equivalent to 
		\begin{equation}
			f'-iVf+i\lambda f=ik\:.
		\end{equation}
		Multiplying both sides by the integrating factor $I_{V,\lambda}(x)$ given in \eqref{eq:3.15a}, this is equivalent to
		\begin{equation}
			(I_{V,\lambda}f)'=iI_{V,\lambda}k\:,
		\end{equation}
		which yields
		\begin{equation} \label{eq:3.11}
			f(x)=\frac{1}{I_{V,\lambda}(x)}\left(i\int_0^x I_{V,\lambda}(t)k(t)dt+C\right)\:,
		\end{equation}
		where $C\in\mathbb{C}$. Using that $I_{V,\lambda}(0)=1$, evaluating \eqref{eq:3.11} at $0$ yields $f(0)=C$ and consequently $C=\rho$. Then, evaluating \eqref{eq:3.11} at $2\pi$ yields the eigenvalue equation
		\begin{equation}
			f(2\pi)=\frac{1}{I_{V,\lambda}(2\pi)}\left(i\int_0^{2\pi}I_{V,\lambda}(t)k(t)dt+\rho\right)=1\:,
		\end{equation}
		which is the same as \eqref{eq:3.4}. Thus, if $\lambda$ satisfies \eqref{eq:3.4}, then it is an eigenvalue of $A_{\rho,k}$. Moreover, for such $\lambda$, the corresponding eigenfunction $\phi_\lambda$ is given by 
		\begin{align}
			\phi_{\lambda}(x)&=\frac{1}{I_{V,\lambda}(x)}\left(i\int_0^x I_{V,\lambda}(t)k(t)dt+\rho\right)\\=&e^{i\int_0^x V(t)dt-i\lambda x}\left(i\int_0^x e^{-i\int_0^t V(s)ds+i\lambda t}k(t)dt+\rho\right)\:.
		\end{align}
		To show that $\dim(\ker(A_{\rho,k}-\lambda))=1$, suppose that for $\lambda\in\sigma(A_{\rho,k})$, there exists two linearly independent functions $f_1, f_2\in \dom(A_{\rho,k})$ with $f_1(2\pi)=f_2(2\pi)=1$ such that 
		\begin{equation}
			A_{\rho,k}f_1=if_1'+Vf_1+k=\lambda f_1\mbox{ and } A_{\rho,k}f_2=if_2'+Vf_2+k=\lambda f_2\:.
		\end{equation}
		Then $h:=f_1-f_2$ is also an eigenfunction of $A_{\rho,k}$ with the same eigenvalue $\lambda$, i.e.\ $A_{\rho,k}h=\lambda h$. Since $h(2\pi)=0=h(0)$, it satisfies 
		\begin{equation}
			(A_{\rho,k}-\lambda)h=ih'+(V-\lambda)h=0
		\end{equation}
		subject to the initial condition $h(0)=0=h(2\pi)$. Consequently $h(x)\equiv 0$, implying $f_1=f_2$, which is a contradiction to $f_1$ and $f_2$ being linearly independent. Hence, $\dim(\ker(A_{\rho,k}-\lambda))=1$.
		To determine the resolvent operator, given $g\in L^2(0,2\pi)$, we explicitly construct the solution $f\in\dom(A_{\rho,k})$ to the equation
		\begin{equation}
			(A_{\rho,k}-\lambda)f=if'+(V-\lambda)f+f(2\pi)k=g\:.
		\end{equation}
		Rearranging this equation and again using the integrating factor $I_{V,\lambda}$, this is equivalent to
		\begin{equation}
			(I_{V,\lambda}f)'=iI_{V,\lambda}f(2\pi)k-iI_{V,\lambda}g\:,
		\end{equation}
		which implies
		\begin{equation} \label{eq:3.17}
			f(x)=\frac{f(2\pi)}{I_{V,\lambda}(x)}\int_0^x iI_{V,\lambda}(t)k(t)dt-\frac{i}{I_{V,\lambda}(x)}\int_0^xI_{V,\lambda}(t)g(t)dt +\frac{C}{I_{V,\lambda}(x)}\:,
		\end{equation}
		where $C\in\mathbb{C}$ is again an integration constant. Evaluating \eqref{eq:3.17} at $0$ -- again using that $I_{V,\lambda}(0)=1$ -- implies 
		\begin{equation} \label{eq:3.18}
			f(0)=C=\rho f(2\pi).
		\end{equation}
		Hence, evaluating \eqref{eq:3.17} at $2\pi$ then yields
		\begin{equation}
			f(2\pi)=\frac{1}{I_{V,\lambda}(2\pi)}\left(\int_0^{2\pi}iI_{V,\lambda}(t)k(t)dt+\rho\right)f(2\pi)-\frac{i}{I_{V,k}(2\pi)}\int_0^{2\pi}I_{V,\lambda}(t)g(t)dt\:,
		\end{equation}
		which can be solved for $f(2\pi)$:
		\begin{align} \label{eq:3.20}
			f(2\pi)&=-\frac{i}{I_{V,\lambda}(2\pi)}\left[1-\frac{1}{I_{V,\lambda}(2\pi)}\left(\int_0^{2\pi}iI_{V,\lambda}(t)k(t)dt+\rho\right)\right]^{-1}\int_0^{2\pi} I_{V,\lambda}(t)g(t)dt\\
			&=\frac{-i\int_0^{2\pi} I_{V,\lambda}(t)g(t)dt}{I_{V,\lambda}(2\pi)-i\int_0^{2\pi} I_{V,\lambda}(t)k(t)dt-\rho}\:,
			\label{eq:3.21}
		\end{align}
		where we are using that since $\lambda\notin \sigma(A_{\rho,k})$, we have
		\begin{equation}
			1-\frac{1}{I_{V,\lambda}(2\pi)}\left(\int_0^{2\pi}iI_{V,\lambda}(t)k(t)dt+\rho\right)\neq 0 
		\end{equation}
		and thus, the expression on the right-hand side of \eqref{eq:3.20} is well-defined. Now, plugging \eqref{eq:3.21} and \eqref{eq:3.18} back into \eqref{eq:3.17} gives
		\eqref{eq:3.27}.
	\end{proof}
	
	\begin{remark} \label{rem:3.2}
		Note that for $\rho=0$ and $k\equiv 0$, the eigenvalue equation \eqref{eq:3.4}
		has no solution. Indeed, the spectrum of the corresponding operator $A_{0,0}$ is empty. On the other hand, observe that for $\lambda\in\varrho(A_{\rho_1,k_1})\cap\varrho(A_{\rho_2,k_2})$ the resolvent difference $(A_{\rho_1,k_1}-\lambda)^{-1}-(A_{\rho_2,k_2}-\lambda)^{-1}$ is at most rank-one. This can be verified directly by considering \eqref{eq:3.27}, but also follows from abstract reasons since $(A_{\rho_1,k_1}-\lambda)^{-1}$ and $(A_{\rho_2,k_2}-\lambda)^{-1}$ coincide on $\ran(A_{min}-\lambda)$, which has codimension one. This is an example of a rank-one perturbation of the resolvents of  unbounded operators whose root functions form Schauder bases (cf.\ Theorem \ref{thm:3.20} below) which yields the resolvent of an operator with empty spectrum. The fact that such a phenomenon can occur just by adding a rank-one perturbation indicates that techniques from the standard perturbation theory of selfadjoint operators will not be applicable to this operator. We also point to \cite{BY15, BY16}, where questions relating to this phenomenon are addressed. 
	\end{remark}
	
	Next, note that the integral kernel of the resolvent $(A_{\rho,k}-\lambda)^{-1}$, which can be read off in \eqref{eq:3.27}, is a bounded function in $L^2((0,2\pi)\times(0,2\pi))$. This implies the following 
	\begin{corollary} \label{coro:3.3}
		For $\lambda\notin\sigma(A_{\rho,k})$, the resolvents $(A_{\rho,k}-\lambda)^{-1}$ are Hilbert-Schmidt.
	\end{corollary}
	
	We now want to gain a better understanding of the nature of the spectrum. To this end, we will apply a similarity transformation to simplify $A_{\rho,k}$.
	
	Thus, we introduce the quantity $\eta=\eta(\rho,V)$:
	\begin{equation} \label{eq:3.7}
		\eta:=\frac{1}{2\pi}\left(\int_0^{2\pi}V(t)dt-i\ln(\rho)\right)\:,
	\end{equation}
	where $\ln(\rho)=\ln|\rho|+i\varphi$ with $\varphi\in[0,2\pi)$ such that $\rho=|\rho|e^{i\varphi}$. Moreover, we introduce the function $W\in H^1(0,2\pi)$ given by
	\begin{equation}  \label{eq:3.32}
		W(x):=\exp\left(-i\int_0^xV(t)dt+i\eta x\right)\:,
	\end{equation}
	as the integrating factor and let $W$ also denote the maximal operator of multiplication by $W(x)$. Note that $W$ is boundedly invertible with its inverse $W^{-1}$ being the maximal operator of multiplication by $1/W(x)$.
	
	It will also be convenient to introduce the operators $P_{1,K}$ given by
	\begin{align}
		P_{1,K}:\quad\dom(P_{1,K})&=\{f\in H^1(0,2\pi)\:|\: f(0)=f(2\pi)\}\\ (P_{1,K}f)(x)&=if'(x)+f(2\pi)K(x)\:.
	\end{align}
	
	For $K\equiv 0$, the corresponding operator $P_{1,0}$ is the selfadjoint periodic momentum operator $P_{1,0}$ given by
	\begin{equation}
		P_{1,0}:\quad\dom(P_{1,0})=\{f\in H^1(0,2\pi)\:|\: f(0)=f(2\pi)\},\quad f\mapsto if'\:.
	\end{equation}
	Its spectrum $\sigma(P_{1,0})=\mathbb{Z}$ is purely discrete and simple. Moreover, for $n\in\mathbb{Z}$, the corresponding eigenfunction $\psi_n$ is given by
	\begin{equation} \label{eq:3.36}
		\psi_n(x)=\frac{1}{\sqrt{2\pi}}e^{-inx}\:,
	\end{equation}
	where it is -- of course -- well-known that $\{\psi_n\}_{n\in\mathbb{Z}}$ forms an orthonormal basis of $\mathcal{H}$.
	
	By direct calculation, one verifies the following
	\begin{lemma}
		Let $k\in L^2(0,2\pi)$ and $K(x):=W(x)k(x)/W(2\pi)$. If $\rho\neq 0$, then the operators $A_{\rho,k}$ and $(P_{1,K}+\eta)$ are similar:
		\begin{equation}
			W^{-1}(P_{1,K}+\eta)W=A_{\rho,k}\:,
		\end{equation}
		where $\eta$ and $W$ are given in \eqref{eq:3.7} and \eqref{eq:3.32}, respectively. 
		\label{lemma:3.4}
	\end{lemma}
	
	Since this implies that $\sigma(A_{\rho,k})=\sigma(P_{1,K}+\eta)$ as well as that $P_{1,K}$ and $A_{\rho,k}$ have the same root functions, we focus on the operators $P_{1,K}$ from now on. In this case, the eigenvalue equation \eqref{eq:3.4} simplifies to
	
	\begin{equation} \label{eq:3.38}
		\Phi(\lambda):=1-e^{2\pi i\lambda}-\int_0^{2\pi}e^{i\lambda t}(-iK(t))dt=0\:.
	\end{equation}
	Since $\int_0^{2\pi}e^{i\lambda t}(-iK(t))dt$ can be thought of as the Fourier transform (in $L^2(\bbR)$) of the compactly supported function $-\sqrt{2\pi}i\chi_{[0,2\pi]}(t)K(t)$, this shows that $\Phi(\lambda)$ is entire. Because of this, the following result is an immediate consequence of the Identity Theorem:
	\begin{theorem} \label{thm:3.5}
		The spectrum of $P_{1,K}$ is purely discrete, i.e. it has no accumulation points.
	\end{theorem}
	While the geometric multiplicity of each eigenvalue of $P_{1,K}$ is equal to one, this is not necessarily the case for the algebraic multiplicity. Indeed, the algebraic multiplicity of an eigenvalue $\lambda\in\sigma(P_{1,K})$ is equal to its order as a zero of $\Phi(\lambda)$:
	\begin{theorem} Let $\lambda\in\sigma(P_{1,K})$. Then the algebraic multiplicity $m_a(\lambda,P_{1,K})$ of $\lambda$ is equal to its multiplicity as a root of $\Phi(\lambda)$. \label{thm:3.6}
	\end{theorem}
	\begin{proof}
		Let $\lambda\in\sigma(P_{1,K})$ with algebraic multiplicity $m_a(\lambda,P_{1,K})=m+1$, where $m\in\mathbb{N}_0$. This implies that $\overline{\lambda}$ is in $\sigma(P_{1,K}^*)$ and moreover, by \cite[\S 4, Chapter I, Result 2.1 on page 10]{GK69} that $m_a(\overline{\lambda},P_{1,K}^*)=m+1$ as well. An explicit calculation shows that the root space $\mathcal{R}(\overline{\lambda},P_{1,K}^*)$ is spanned by the functions $\{e^{-i\overline{\lambda}x},xe^{-i\overline{\lambda}x},\dots,x^me^{-i\overline{\lambda}x}\}$, which all have to be elements of $\dom(P_{1,K}^*)$, i.e. they have to satisfy the boundary condition 
		\begin{equation} \label{eq:3.39}
			f(2\pi)=f(0)+\langle iK,f\rangle\:.
		\end{equation}
		It can be directly verified that the actual eigenfunction $\tfrac{1}{\sqrt{2\pi}}e^{-i\overline{\lambda}x}$ satisfying this condition is equivalent to $\Phi(\lambda)=0$. Now, let $\phi_k(x)=\tfrac{1}{\sqrt{2\pi}}x^ke^{-i\overline{\lambda}x}$ be a root function. Then, Condition \eqref{eq:3.39} being satisfied by $\phi_k$ is equivalent to
		\begin{equation} \label{eq:3.40}
			(2\pi)^ke^{-i2\pi\overline{\lambda}}=\int_0^{2\pi}\overline{iK(t)}t^ke^{-i\overline{\lambda}t}dt\:.
		\end{equation}
		On the other hand, computing the $k$-th derivative of $\Phi$ and evaluating it at $\lambda$ yields
		\begin{equation} \label{eq:3.41}
			\Phi^{(k)}(\lambda)=-(2\pi i)^ke^{2\pi i\lambda}-\int_0^{2\pi}(it)^ke^{i\lambda t}(-iK(t))dt\:.
		\end{equation}
		By taking the complex conjugate of \eqref{eq:3.40} and comparing it to \eqref{eq:3.41}, it is easily seen that $\phi_k\in\dom(P_{1,K}^*)$ if and only if $\Phi^{(k)}(\lambda)=0$, i.e. if $\lambda$ is a root of $\Phi$ with multiplicity $k$. This shows the theorem.
	\end{proof}
	Since we have shown that the algebraic multiplicity of any eigenvalue is given by the multiplicity of a root of an entire function, we conclude
	\begin{corollary}
		All root spaces of $P_{1,K}$ and $P_{1,K}^*$ are finite-dimensional.
	\end{corollary}
	It turns out that the zeros of functions of the form $\Phi(\lambda)$ have been studied in different contexts before in the 1960's \cite{AB62, C64}. Most useful for our purposes will be the following result by Anselone/Boas and Cartwright:
	\begin{proposition} \label{prop:2.4} Let $\Psi(z)$ be given by
		\begin{equation} \label{eq:2.39}
			\Psi(z)=1-e^z-\int_{-1}^1 e^{-zs}\hat{K}(s)ds
		\end{equation}
		for some $\hat{K}\in L^1(-1,1)$. Then, for each $\varepsilon>0$ and each $\nu=0,\pm 1,\pm 2, \dots$ let $\Gamma_\varepsilon(2\nu \pi i)$ denote the circle with center $2\nu\pi i$ and radius $\varepsilon$. For each $\varepsilon, 0 < \varepsilon <1/2$, there exists $\nu_\varepsilon$ such that if $|\nu|\geq \nu_\varepsilon$, then $\Psi(z)$ has just one zero inside $\Gamma_\varepsilon(2\nu\pi i)$. Furthermore, each of these zeros is simple, i.e.\ of order one. Moreover, $\Psi(z)$ has only finitely many additional zeros if and only if $\hat{K}(s)=0$ for almost every $s\in(0,1)$. 
	\end{proposition}
	Applying this now to the function $\Phi(\lambda)$ implies the following result about the spectrum of $P_{1,K}$:
	\begin{theorem} \label{thm:3.7}
		For every $\varepsilon>0$, there exists $N_\varepsilon\in\mathbb{N}$ such that if $n\in\mathbb{Z}$, $|n|\geq N_\varepsilon$, there is exactly one eigenvalue $\lambda\in\sigma(P_{1,K})$ with $|\lambda-n|<\varepsilon$. Moreover, $P_{1,K}$ has only finitely many additional eigenvalues.
	\end{theorem}
	\begin{proof} By Theorem \ref{thm:3.1}, the spectrum of $P_{1,K}$ consists only of eigenvalues and is given by the zeros of $\Phi(\lambda)$. For $z\in\mathbb{C}$, consider the function
		\begin{align}
			\Phi\left(\frac{z}{2\pi i}\right)=&1-e^z-\int_0^{2\pi}e^{\frac{z}{2\pi}t}(-iK(t))dt\\=&1- e^z-\int_{-1}^0e^{-zs}\left(-2\pi i K(-2\pi s)\right) ds   \:,
			\label{eq:2.41}
		\end{align}
		where the last equality follows from a simple substitution. Since $K(x)=W(x)k(x)/W(2\pi)$, where $k\in L^2(0,2\pi)$, this also implies that $K\in L^2(0,2\pi)$ and consequently, the function 
		\begin{equation} \label{eq:2.42}
			\hat{K}(s)=\begin{cases}-2\pi iK(-2\pi s)\quad&\mbox{if}\quad s\in(-1,0)\\ 0 \quad&\mbox{if}\quad s\in(0,1)\end{cases}
		\end{equation}
		is in $L^1(-1,1)$ and satisfies $\hat{K}(s)=0$ for almost every $s\in(0,1)$. Hence, we can apply Proposition \ref{prop:2.4} to $\Psi(z)=\Phi(\frac{z}{2\pi i})$ with $\hat{K}$ in \eqref{eq:2.39} being the function given by \eqref{eq:2.42}. The statement about the zeros of $\Phi(\lambda)$, which are the eigenvalues of $P_{1,K}$, and thus the theorem follows.
	\end{proof}
	
	We further refine this result in two directions: $(i)$ in Theorem \ref{thm:3.11a}, we will give a more quantitative estimate about the radii of the circles around $n\in\mathbb{Z}$ such that we can find a unique simple eigenvalue of $P_{1,K}$ inside it and $(ii)$ in Theorem \ref{thm:3.13}, we will give a more quantitative statement about the finitely many other eigenvalues that may not be found inside these small circles. 
	
	To this end, we need to introduce \emph{doubly infinite Hilbert matrices} which are bounded operators in $\ell^2(\mathbb{Z})$ (cf.\ \cite[Thm.\ 2]{Hill61}):
	
	\begin{definition}
		Let $\eta\in\mathbb{C}\setminus\mathbb{Z}$. Then the doubly infinite Hilbert matrix $\mathcal{T}_\eta$ is the operator in $\ell^2(\mathbb{Z})$ given by
		\begin{equation}
			\mathcal{T}_\eta:\quad\ell^2(\mathbb{Z})\rightarrow\ell^2(\mathbb{Z}),\quad (\mathcal{T}_\eta a)_n=\sum_{m\in\mathbb{Z}}\frac{1}{m+n+\eta}a_m\:.
		\end{equation}
	\end{definition}
	We then get the following two theorems:
	\begin{theorem} \label{thm:3.11a}
		Fix $K\in L^2(0,2\pi)$. Then, there exists $N=N(K)\in\mathbb{N}$ such that for every $n\in\mathbb{Z}$ with $|n|\geq N$ and $\langle \psi_n,K\rangle\neq 0$, the function $\Phi$ has exactly one simple root inside the disk with center point $n$ and radius $|\langle \psi_n,K\rangle|$. Moreover, for every $n$ with $|n|\geq N$ and $\langle \psi_n,K\rangle=0$, the number $n$ is a simple zero of $\Phi$ and there are no other zeros of $\Phi$ in any disk with center point $n$ and radius less than $1/2$.
	\end{theorem}
	\begin{proof} 
		Let $n\in\mathbb{Z}$ and $\eta\in\mathbb{C}$. Then, decomposing $\Phi(\lambda)=f(\lambda)+g(\lambda)$, where 
		\begin{equation}
			f(\lambda)= 1-e^{2\pi i \lambda}\quad\mbox{and}\quad g(\lambda)=i\int_0^{2\pi} e^{i\lambda t}K(t)dt\:,
		\end{equation}
		note that $|f(n+\eta)|=|1-e^{2\pi i\eta}|$, while for $|g(n+\eta)|$, we obtain
		\begin{align}
			&|g(n+\eta)|=\left|\int_0^{2\pi}e^{i(n+\eta) t}K(t)dt\right|=|\langle e^{-i(n+\overline{\eta})t},K\rangle|\\=&\left|\sum_{m=-\infty}^\infty\langle e^{-i(n+\overline{\eta})t},\psi_{-m}\rangle\langle \psi_{-m},K\rangle\right|\\=&\frac{1}{\sqrt{2\pi}}\left|\sum_{m=-\infty}^\infty \left(\int_0^{2\pi}e^{i(n+\eta+m)t}dt\right)\langle \psi_{-m},K\rangle\right|\\
			=&\frac{1}{\sqrt{2\pi}}\left|1-e^{2\pi i \eta}\right|\left|\sum_{m=-\infty}^\infty \frac{1}{n+m+\eta}(a(K))_m\right|\\
			=&|f(n+\eta)| \frac{1}{\sqrt{2\pi}} \left|(\mathcal{T}_\eta a(K) )_n\right|\:, \label{eq:3.52}
		\end{align}
		where $a(K)\in\ell^2(\mathbb{Z})$ is the sequence defined by $(a(K))_m:=\langle \psi_{-m},K\rangle$.

		In what follows, we will abbreviate $a:=a(K)$. Then, note that for any $\eta\in\mathbb{C}$ with $|\eta|=\varepsilon<1/2$ we have
		\begin{align}
			&|(\mathcal{T}_\eta a)_n|\label{eq:3.53}\\&=\left|\sum_{m=-\infty}^\infty \frac{1}{n+m+\eta}a_m\right|\leq \frac{|a_{-n}|}{\varepsilon}+\left|\sum_{m\neq -n}\frac{a_m}{n+m+\eta}\right|\\&\leq \frac{|a_{-n}|}{\varepsilon}+\left|\sum_{m\neq -n}\frac{(1/2-\eta)a_m}{(n+m+\eta)(n+m+1/2)}\right|+\left| \sum_{m\neq -n}\frac{a_m}{n+m+1/2}\right|\\
			&= \frac{|a_{-n}|}{\varepsilon}+\left|\sum_{m\neq -n}\frac{(1/2-\eta)a_m}{(n+m+\eta)(n+m+1/2)}\right|+|(\mathcal{T}_{1/2}a)_n-2a_{-n}|\\&\leq \left(2+\frac{1}{\varepsilon}\right)|a_{-n}|+|(\mathcal{T}_{1/2}a)_n|+|1/2-\eta|\sum_{m\neq -n}\frac{|a_m|}{(|n+m|-\varepsilon)(|n+m|-1/2)}\\&\leq \left(2+\frac{1}{\varepsilon}\right)|a_{-n}|+|(\mathcal{T}_{1/2}a)_n|+\sum_{m\neq -n}\frac{|a_m|}{(|n+m|-1/2)^2}\\&\leq \left(2+\frac{1}{\varepsilon}\right)|a_{-n}|+|(\mathcal{T}_{1/2}a)_n|+\sum_{m\in\mathbb{Z}}\frac{|a_m|}{(|n+m|-1/2)^2}\:. \label{eq:3.59}
		\end{align}
		Let us now individually focus on the three terms that occur in \eqref{eq:3.59}. Since $a\in\ell^2(\mathbb{Z})$ and $(\mathcal{T}_{1/2}a)\in \ell^2(\mathbb{Z})$, there exists an $N_1=N_1(K)\in\mathbb{N}$ such that $2|a_{-n}|+|(\mathcal{T}_{1/2}a)_n|<\frac{1}{2}(\sqrt{2\pi}-1)$ for all $n\in\mathbb{Z}$ with $|n|\geq N_1$. 
		
		Next, we focus on the behavior of the last term in \eqref{eq:3.59} as $|n|\rightarrow\infty$:
		\begin{align} \label{eq:3.61}
			&\lim_{n\rightarrow\pm\infty}\sum_{m\in\mathbb{Z}}\frac{|a_m|} {(|n+m|-1/2)^2}=\lim_{n\rightarrow\pm\infty}\sum_{m\in\mathbb{Z}}\frac{|a_{m-n}|}{(|m|-1/2)^{2}}\\&=\sum_{m\in\mathbb{Z}}\left(\frac{\lim_{n\rightarrow\pm\infty}|a_{m-n}|}{(|m|-1/2)^{2}}\right)=0\:,
		\end{align}
		since $a$ is an $\ell^2(\mathbb{Z})$ sequence. Interchanging the limit and summation is justified by dominated convergence, since for every $n\in\mathbb{Z}$ we have the estimate
		\begin{equation}
			\frac{|a_{m-n}|}{(|m|-1/2)^{2}}\leq \frac{\|a\|_\infty}{(|m|-1/2)^2}\:, \label{eq:3.63}
		\end{equation}
		which is summable. Consequently, there exists $N_2=N_2(K)\in\mathbb{N}$ such that the third term in \eqref{eq:3.59} satisfies
		\begin{equation}
			\sum_{m\in\mathbb{Z}}\frac{|a_m|}{(|n+m|-1/2)^2}< \frac{1}{2}(\sqrt{2\pi}-1)
		\end{equation}
		for all $n\in\mathbb{Z}$ with $|n|\geq N_2$. Assuming $|a_{-n}|\neq0$, we choose $\varepsilon=|a_{-n}|$ and further estimate \eqref{eq:3.59}:
		\begin{equation}
			|(T_\eta a)_n|\leq \eqref{eq:3.59}<  \frac{1}{2}(\sqrt{2\pi}-1)+1+\frac{1}{2}(\sqrt{2\pi}-1)=\sqrt{2\pi}\:,
		\end{equation}
		for any $n\in\mathbb{Z}$ such that $|n|\geq \max\{N_1,N_2,N_3\}=:N$, where $N_3=N_3(K)$ is chosen such that $\varepsilon=|a_{-n}|<1/2$, for all $n\in\mathbb{Z}$ with $|n|\geq N_3$, which was needed to estimate the term $|(\mathcal{T}a)_n|$ in \eqref{eq:3.53} and after.
		
		Hence, with the choice $a=a(K)$, this shows that if $|n|\geq N$, we have 
		\begin{equation}
			|g(n+\eta)|=|f(n+\eta)|\frac{1}{\sqrt{2\pi}}|(\mathcal{T}_\eta a(K))_n|< |f(n+\eta)|\:.
		\end{equation}
		Since the roots (all of them simple) of $f$ are given by the integers, it follows from Rouch\'e's theorem that $\Phi$ has exactly one simple root inside the disk of radius $|a_{-n}(K)|=|\langle \psi_n,K\rangle|$ with center point $n$. 
		
		Next, consider the case $a_{-n}=0$. By direct inspection, it can be seen that $n$ itself is a root of $\Phi$, i.e.\ $\Phi(n)=0$. Moreover, the estimate \eqref{eq:3.59} simplifies to
		\begin{equation}
			|(\mathcal{T}_\eta a)_n|\leq |(\mathcal{T}_{1/2}a)_n|+\sum_{m\in\mathbb{Z}}\frac{|a_m|}{(|n+m|-1/2)^2}
		\end{equation}
		for any $\eta$ with $|\eta|<1/2$. By the same reasoning as above, for $|n|\geq N$, we have $|(\mathcal{T}_\eta a)_n|<\sqrt{2\pi}$, and consequently $|g(n+\eta)|<|f(n+\eta)|$ for any $\eta$ with $|\eta|<1/2$. Hence, $\Phi$ has no other root  inside any disk of radius less than $1/2$. 
		
		This finishes the proof.
	\end{proof}
	It is, of course, well-known that smoothness properties of a function $K\in L^2(0,2\pi)$ can be related to the decay of its Fourier coefficients, see, e.~g.~ \cite[I.4, I.5]{Katz04}. As an example of a simple application of the previous theorem, we deduce the following corollary for the case that $K\in\dom((P_{1,0})^{j})$ for some $j\in\mathbb{N}$.
	\begin{corollary} Let $K\in\dom((P_{1,0})^{j})$ for some $j\in\mathbb{N}$. Then, the radii $r_n$ of the disks with center point $n$, inside which exactly one simple root of $\Phi(\lambda)$ can be found, satisfy $r_n=o(|n|^{-j})$ as $n\rightarrow\pm \infty$.
	\end{corollary}
	\begin{proof}
		Since $K\in\dom(P_{1,0})$, we get for $n\neq 0$:
		\begin{align}
			(a(K))_n&=\langle \psi_{-n},K\rangle=\frac{(-1)^j}{n^j}\langle (P_{1,0})^j\psi_{-n},K\rangle\\=&\frac{(-i)^j}{n^j}\langle \psi_{-n}, K^{(j)}\rangle =\frac{1}{(in)^j} (a(K^{(j)}))_n
		\end{align}
		and consequently, we have
		\begin{equation}
			|(a(K))_n|\leq \frac{1}{1+|n|^j}|(a(K^{(j)}))_n|
		\end{equation}
		for every $n\in\mathbb{Z}$. Since $(a(K^{(j)}))_n\in\ell^2(\mathbb{Z})$, we have $\lim_{n\rightarrow\pm \infty}|a(K^{(j)})_n|=0$ and the corollary follows by an application of Theorem \ref{thm:3.11a}.
	\end{proof}
	\begin{theorem}
		Let $K\in L^2(0,2\pi)$ be fixed. Then, there exist $N\in\mathbb{N}$ and $b_0>0$ such that for all $n\in\mathbb{N}$ with $n\geq N$ and all $b\geq b_0$, the function $\Phi(\lambda)$ has precisely $2n+1$ zeros (counting multiplicity) inside the rectangle
		\begin{equation}
			\mathcal{R}_{n,b}:=\{z\in\mathbb{C}\:|\: -(n+1/2)\leq \Re(z) \leq n+1/2,\: -b\leq \Im z\leq b\}\:. 
		\end{equation}
		\label{thm:3.13}
	\end{theorem}
	\begin{proof} Again, we write $\Phi(\lambda)=f(\lambda)+g(\lambda)$, where 
		\begin{equation}
			f(\lambda)= 1-e^{2\pi i \lambda}\quad\mbox{and}\quad g(\lambda)=i\int_0^{2\pi} e^{i\lambda t}K(t)dt
		\end{equation}
		and note that, since the roots of $f(\lambda)$ are given by the integers, $f(\lambda)$ has exactly $2n+1$ zeros in $\mathcal{R}_{n,b}$ for any $n\in\mathbb{N}$ and any $b>0$. The claimed result will follow from Rouch\'e's theorem if we can show $|g(\lambda)|<|f(\lambda)|$ on $\partial \mathcal{R}_{n,b}$ for large enough rectangles $\mathcal{R}_{n,b}$. 
		
		To this end, we begin with the horizontal segments of $\partial\mathcal{R}_{n,b}$, i.e.\ we will show that there is a $b_0>0$ such that $|g(x+iy)|<|f(x+iy)|$  for all $x\in\mathbb{R}$ and all $y\in\mathbb{R}$ with $|y|\geq b_0$. We begin with estimating $|g(x+iy)|$:
		\begin{align}
			|g(x+iy)|=&\left|\int_0^{2\pi}e^{ixt}e^{-yt}K(t)dt\right|\leq \int_0^{2\pi}e^{-yt}|K(t)|dt\leq \|K\|\left(\int_0^{2\pi}e^{-2yt}dt\right)^{1/2}\\=&\|K\|\left(\frac{1}{2y}\left(1-e^{-4\pi y}\right)\right)^{1/2}=\|K\|\frac{|1-e^{-4\pi y}|^{1/2}}{|2y|^{1/2}} \label{eq:3.46}\:,
		\end{align}
		where we used Cauchy-Schwarz for the second inequality. 
		Next, we look at $|f(x+iy)|$:
		\begin{align}
			|f(x+iy)|=&\left|1-e^{2\pi ix }e^{-2\pi y}\right|=\left(1+e^{-4\pi y}-2\cos(2\pi x)e^{-2\pi y}\right)^{1/2}\\\geq & \left(1-2e^{-2\pi y}+e^{-4\pi y}\right)^{1/2}=\left|1-e^{-2\pi y}\right| \label{eq:3.48}\:.
		\end{align}
		Now, clearly there exists a $b_0>0$, where $b_0=b_0(\|K\|)$, such that $\|K\|\frac{|1-e^{-4\pi y}|^{1/2}}{|2y|^{1/2}}<\left|1-e^{-2\pi y}\right|$ for all $|y|\geq b_0$. It then follows from \eqref{eq:3.46} and \eqref{eq:3.48} that $|f(x+ib)|>|g(x+ib)|$ as well as $|f(x-ib)|>|g(x-ib)|$ for all $x\in\mathbb{R}$ and all $b\geq b_0$. 
		
		For what follows, let $b\geq b_0$ be fixed. 
		
		We then need to show $|g(\lambda)|<|f(\lambda)|$ for $\lambda$ on the vertical segments of $\partial\mathcal{R}_{n,b}$, i.e.\ we want to show that there exists $N\in\mathbb{N}$ such that $|f(\pm(n+1/2)+iy)|>|g(\pm(n+1/2)+iy)|$ for all $n\in\mathbb{N}$ with $n\geq N$ and all $y\in\mathbb{R}$ such that $|y|\leq b$. We will only consider the case $\lambda=n+1/2+iy$, the case $\lambda=-n-1/2+iy$ follows analogously. As done in \eqref{eq:3.52}, we get
		\begin{equation}
			|g(n+1/2+iy)|=|f(n+1/2+iy)|\frac{1}{\sqrt{2\pi}}|(\mathcal{T}_{1/2+iy}a(K))_n|\:,
		\end{equation}
		so we need to show that there exists $N\in\mathbb{N}$ such that $\frac{1}{\sqrt{2\pi}}|(\mathcal{T}_{1/2+iy}a(K))_n|<1$ for all $y\in\mathbb{R}$ with $|y|\leq b$ and all $n\geq N$. Similarly as in the proof of Theorem \ref{thm:3.11a}, we further estimate -- again abbreviating $a=a(K)$ --
		\begin{align}
			&|(\mathcal{T}_{1/2+iy}a)_n|\\
			=&\left|\sum_{m\in\mathbb{Z}}\frac{1}{n+m+1/2+iy}a_m\right|\\
			\leq & \left|\sum_{m\in\mathbb{Z}}\frac{-iy}{(n+m+1/2)(n+m+1/2+iy)}a_m\right|+\left|\sum_{m\in\mathbb{Z}}\frac{1}{n+m+1/2}a_m\right|\\
			\leq & \left(\sum_{m\in\mathbb{Z}}\frac{|y|}{|n+m+1/2||n+m+1/2+iy|}|a_m|\right)+|(\mathcal{T}_{1/2}a)_n|\\
			\leq & \left(\sum_{m\in\mathbb{Z}}\frac{b}{(|n+m|-1/2)^2}|a_m|\right)+|(\mathcal{T}_{1/2}a)_n|\overset{n\rightarrow\infty}{\longrightarrow}0\:, \label{eq:3.81}
		\end{align}
		where $(\mathcal{T}_{1/2}a)_n$ goes to zero since $(\mathcal{T}_{1/2}a)\in\ell^2(\mathbb{Z})$ and it was already argued in the proof of Theorem \ref{thm:3.11a}, Equations \eqref{eq:3.61} to \eqref{eq:3.63}, that the first term in \eqref{eq:3.81} goes to zero as $n\rightarrow\infty$. 
		
		This shows that there exists $N\in\mathbb{N}$ such that 
		$|g(\pm(n+1/2)+iy)|< |f(\pm(n+1/2)+iy)|$ for all $|n|\geq N$ and all $y\in\mathbb{R}$ with $|y|\leq b$.
		
		Hence, if $b\geq b_0$ and $|n|\geq N$, we have $|g(\lambda)|<|f(\lambda)|$ on $\partial\mathcal{R}_{n,b}$ and thus, $\mathcal{R}_{n,b}$ contains exactly $(2n+1)$ roots of $\Phi(\lambda)$, counting multiplicities. This shows the theorem.
	\end{proof}
	
	We summarize the findings of this section in the following theorem:
	\begin{theorem}
		Let $K\in L^2(0,2\pi)$ be fixed. Then, there exists $N_0\in\mathbb{N}$ and a $b_0>0$ such that for every $N\geq N_0$ and every $b\geq b_0$
		\begin{itemize}
			\item[(i)] The rectangle $\mathcal{R}_{N,b}$ contains exactly $(2N+1)$ eigenvalues of $P_{1,K}$, counting algebraic multiplicity.
			\item[(ii)] For each $n\in\mathbb{Z}$ with $|n|>N$ and $\langle \psi_n,K\rangle\neq 0$, there is exactly one eigenvalue of $P_{1,K}$ of algebraic multiplicity one in the disk with center point $n$ and radius $|\langle \psi_n,K\rangle|$. 
			\item[(iii)] For each $n\in\mathbb{Z}$ with $|n|>N$ and $\langle \psi_n,K\rangle=0$, the number $n$ is an eigenvalue of $P_{1,K}$ of algebraic multiplicity one.
		\end{itemize}
		Moreover, this describes all eigenvalues of $P_{1,K}$. \label{thm:3.14}
	\end{theorem}
	\begin{proof}
		
		By Theorem \ref{thm:3.7}, there are only finitely roots of $\Phi$ (counting multiplicity) which are not inside disks around integers of large modulus with small radii. Choose $N_0'\in\mathbb{N}$ and $b_0'>0$ large enough such that $\mathcal{R}_{N_0',b_0'}$ contains all of these roots. Moreover, by Theorem \ref{thm:3.13}, there exist $N_0''\in\mathbb{N}$ and $b_0''>0$ such that for every $N\geq N_0''$ and every $b\geq b_0''$, the function $\Phi$ has exactly $(2N+1)$ roots (counting multiplicities) inside $\mathcal{R}_{N,b}$. Lastly, by Theorem \ref{thm:3.11a}, we can choose $N_0'''\in\mathbb{N}$ such that for every $n\in\mathbb{Z}$ with $|n|> N_0'''$ and $\langle \psi_n,K\rangle\neq 0$, there is exactly one simple root of $\Phi$ inside the disk with center point $n$ and radius $|\langle \psi_n,K\rangle|$ and if $\langle \psi_n,K\rangle=0$, then $n$ is a simple root of $\Phi$. 
		
		Since by Theorem \ref{thm:3.6}, the number of roots of $\Phi(\lambda)$ (counting multiplicity) inside a domain corresponds to the numbers of eigenvalues of $P_{1,K}$ inside the same domain (counting algebraic multiplicities), the result then follows when taking $N_0:=\max\{N_0', N_0'', N_0'''\}$ and $b_0=\max\{b_0', b_0''\}$.
	\end{proof}
	
	Figure \ref{fig:1} is an illustration of this result for the specific choice $\widetilde{K}(t)=\frac{1-i}{2}(t-\pi)$. The eigenvalues of $P_{1,\widetilde{K}}$ are plotted in red. It shows the rectangle $\mathcal{R}_{3,1.4}$, which contains the seven eigenvalues of $P_{1,\widetilde{K}}$ that lie closest to zero in the complex plane. The circles around $-6,-5,-4$, and $4,5,6$ have radii equal to the absolute values of the corresponding Fourier coefficients and each contains exactly one eigenvalue of $P_{1,\widetilde{K}}$. For this particular example note that the algebraic multiplicity of each eigenvalue in $\sigma(P_{1,\widetilde{K}})$ is equal to one. Lastly, observe that since $\widetilde{K}\perp\psi_0$, the number $0$ is an eigenvalue of $P_{1,\tilde{K}}$. 
	
	\begin{figure}[htbp]
		\centering
		\includegraphics[clip, trim=2.5cm 10cm 2.5cm 10cm, width=1.00\textwidth]{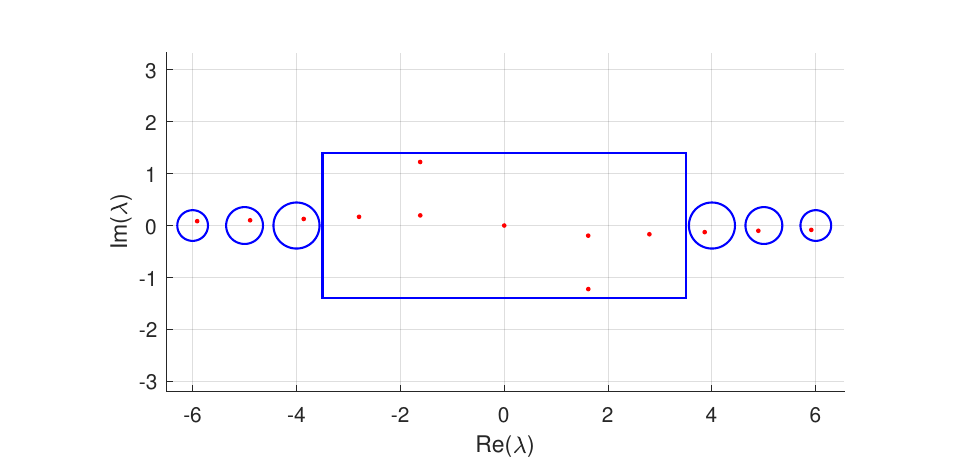}
		\caption{Eigenvalues of $P_{1,\widetilde{K}}$ for $\widetilde{K}(t)=\frac{1-i}{2}(t-\pi)$. The recatngle and circles illustrate the result stated in Theorem \ref{thm:3.14}.}
		\label{fig:1}
	\end{figure}

	\subsection{Riesz basis property of root functions}

	From \eqref{eq:3.6}, note that the adjoint $P_{1,K}^*$ is given by
	\begin{equation}
		P_{1,K}^*:\quad\dom(P_{1,K}^*)=\{f\in H^1(0,2\pi)\:|\: f(2\pi)=f(0)+\langle iK,f\rangle\},\quad f\mapsto if'\:.
	\end{equation}

	Next, we will show that the root vectors of $P_{1,K}$ and $P_{1,K}^*$ both form Riesz bases. We will need the following definitions and previous results:
	
	\begin{definition}{\cite[p.~316]{GK69}} Two sequences of vectors $\{f_n\}_{n\in\mathbb{Z}}, \{g_n\}_{n\in\mathbb{Z}}$ in a Hilbert space $\mathcal{H}$ are said to be \emph{quadratically close} if
		\begin{equation}
			\sum_{n\in\mathbb{Z}}\|f_n-g_n\|^2<\infty\:.
		\end{equation}
	\end{definition}
	\begin{definition}{\cite[p.\ 316]{GK69}} A sequence of vectors $\{g_n\}_{n\in\mathbb{Z}}$ is said to be $\omega$-\emph{linearly independent} if the equality 
		\begin{equation}
			\sum_{n\in\mathbb{Z}}c_ng_n=0
		\end{equation}
		is not possible for 
		\begin{equation}
			0<\sum_{n\in\mathbb{Z}}|c_n|^2\|g_n\|^2<\infty\:.
		\end{equation}
	\end{definition}
	\begin{proposition}{\cite{B51}; presentation given in \cite[Theorem VI.2.3]{GK69}} Any $\omega$-linearly independent sequence $\{g_n\}_{n\in\mathbb{Z}}$ which is quadratically close to some Riesz basis $\{f_n\}_{n\in\mathbb{Z}}$, is also a Riesz basis. \label{prop:3.18}
	\end{proposition}
	The next result follows from an argument given in {\cite[\S 4, Chapter VI, Result 1 on page 329]{GK69}}. 
	\begin{lemma}{\cite[p.\ 316]{GK69}}  Let $\{f_n\}_{n\in\mathbb{Z}}$ be a system made up of bases of all the root subspaces of either $P_{1,K}$ or $P^*_{1,K}$. Then $\{f_n\}_{n\in\mathbb{Z}}$ is $\omega$-linearly independent. \label{lemma:3.19}
	\end{lemma}
	\begin{proof}
		We will only argue for the operator $P_{1,K}$, the result for $P_{1,K}^*$ follows analogously. 
		
		To begin with, let $\mu\in\varrho(P_{1,K})$. Firstly, observe that by Corollary \ref{coro:3.3}, the resolvent $(P_{1,K}-\mu)^{-1}$ is compact and therefore $\sigma((P_{1,K}-\mu)^{-1})$ is given by
		\begin{equation}
			\sigma((P_{1,K}-\mu)^{-1})=\left\{\frac{1}{\lambda-\mu}\:\bigg|\: \lambda\in\sigma(P_{1,K})\right\}\cup\{0\}
		\end{equation}
		and for every $\lambda\in\sigma(P_{1,K})$, we have
		\begin{equation}
			\mathcal{R}(\lambda,P_{1,K})=\mathcal{R}\left(\frac{1}{\lambda-\mu}, (P_{1,K}-\mu)^{-1}\right)\:.
		\end{equation}
		Therefore, if $\{f_n\}_{n\in\mathbb{Z}}$ is a system made up of bases of all the root subspaces of $P_{1,K}$, it is also the system made up of bases of all the root subspaces of $(P_{1,K}-\mu)^{-1}$ corresponding to all non-zero eigenvalues.
		
		Noting that by Corollary \ref{coro:3.3}, the resolvent $(P_{1,K}-\mu)^{-1}$ is compact, for every $z_0\in\sigma((P_{1,K}-\mu)^{-1})\setminus\{0\}$, there is a $\delta>0$ such that there are no other eigenvalues of $(P_{1,K}-\mu)^{-1}$ inside the disk $|z-z_0|=\delta$. Then, the operator
		\begin{equation}
			\mathscr{P}(z_0)=-\frac{1}{2\pi i}\ointctrclockwise_{|z-z_0|=\delta} \left[(P_{1,K}-\mu)^{-1}-z\right]^{-1}dz\:,
		\end{equation}
		where the integration contour is in counter-clockwise direction, is a (not necessarily orthogonal) projector (Riesz projection) onto $\mathcal{R}(z_0,(P_{1,K}-\mu)^{-1})$, such that
		\begin{equation}
			\ran\left(\mathscr{P}(z_0)\right)=\mathcal{R}(z_0,(P_{1,K}-\mu)^{-1})
		\end{equation}
		and
		\begin{equation}
			\mathscr{P}(z_0)\mathcal{R}(z_1,(P_{1,K}-\mu)^{-1})=0\quad\mbox{for every}\quad z_1\in\sigma((P_{1,K}-\mu)^{-1})\setminus\{0,z_0\}\:,
		\end{equation}
		cf.\ Section I.1 and I.2 in \cite{GK69}. Moreover, $\mathcal{H}$ can be decomposed into the direct sum
		\begin{equation}
			\mathcal{H}=\mathcal{R}(z_0,(P_{1,K}-\mu)^{-1})\dot{+}\mathfrak{N}_{z_0}\:,
		\end{equation}
		where both $\mathcal{R}(z_0,(P_{1,K}-\mu)^{-1})$ and $\mathfrak{N}_{z_0}=\ran(I-\mathscr{P}(z_0))$ are closed subspaces of $\mathcal{H}$ that are invariant under $(P_{1,K}-\mu)^{-1}$.
		
		Now assume that
		\begin{equation}
			\sum_{n\in\mathbb{Z}}c_nf_n=0\quad\mbox{and}\quad 0<\sum_{n\in\mathbb{Z}}|c_n|^2\|f_n\|^2<\infty\:, \label{eq:3.117}
		\end{equation}
		where for each $n\in \mathbb{Z}$, there exists $z_n\in\sigma((P_{1,K}-\mu)^{-1})\setminus\{0\}$ such that $f_n\in\mathcal{R}(z_n,(P_{1,K}-\mu)^{-1})$. Since each root space is finite-dimensional, we may assume without loss of generality that $z_{m_1}\neq z_{m_2}$ for $m_1\neq m_2$. Then, applying $\mathscr{P}(z_m)$ yields
		\begin{equation}
			0=\mathscr{P}(z_m)\sum_{n\in\mathbb{Z}}c_nf_n=c_m\mathscr{P}(z_m)f_m+\mathscr{P}(z_m)\left(\sum_{n\neq m}c_nf_n\right)=c_mf_m\:,
		\end{equation}
		where we note that $\sum_{n\neq m}c_mf_m\in\mathfrak{N}_{z_m}$, since for every $n\neq m$, we have $f_n\in\mathcal{R}(z_n,(P_{1,K}-\mu)^{-1})\subseteq \mathfrak{N}_{z_m}$ and $\mathfrak{N}_{z_n}$ is a closed subspace of $\mathcal{H}$. Consequently, we have $c_m=0$, and since we can do this for any $m\in\mathbb{Z}$, this implies $c_m=0$ for all $m\in\mathbb{Z}$, which is a contradiction to the right-hand side of \eqref{eq:3.117}. This shows the lemma.
		
	\end{proof}
	
	We are now prepared to show our main result:
	\begin{theorem}
		For any $K\in L^2(0,2\pi)$, the root vectors of $P_{1,K}$ and of $P_{1,K}^*$ form Riesz bases of $L^2(0,2\pi)$.
		\label{thm:3.20}
	\end{theorem}
	\begin{proof}
		Since by Lemma \ref{lemma:3.19}, the root vectors of $P_{1,K}$ and $P_{1,K}^*$ are $\omega$-linearly independent, we only need to show that they are quadratically close to the orthogonal basis $\{\psi_n\}_{n\in\mathbb{Z}}$. Then it will follow from Proposition \ref{prop:3.18} that they form Riesz bases of $L^2(0,2\pi)$. We will only prove that the root functions of $P_{1,K}$ form a Riesz basis of $L^2(0,2\pi)$ and explain in Remark \ref{rem:3.21} how the proof needs to be modified when considering the root functions of $P_{1,K}^*$ instead. 
		
		Firstly, by Theorem \ref{thm:3.14}, we can pick a rectangle $\mathcal{R}_{N,b}$ such that it contains exactly $(2N+1)$ eigenvalues $\{\lambda_n\}_{n=-N}^{N}$ of $P_{1,K}$ (counting algebraic multiplicity) and all other eigenvalues $\{\lambda_n\}_{|n|>N}$ of $P_{1,K}$ have algebraic multiplicity equal to one and moreover satisfy $|\lambda_n-n|\leq |\langle \psi_n,K\rangle|\leq \tfrac{1}{4\pi}$.
		
		Let $\phi_{\lambda_n}$ denote the root function/eigenfunction of $P_{1,K}$ corresponding to $\lambda_n$. We then get
		\begin{equation}
			\sum_{n\in\mathbb{Z}}\|\psi_n-\phi_{\lambda_n}\|^2=\sum_{n=-N}^N\|\psi_n-\phi_{\lambda_n}\|^2+\sum_{|n|>N}\|\psi_n-\phi_{\lambda_n}\|^2\:,
		\end{equation}
		which means we have to show that $\sum_{|n|>N}\|(2\pi)^{-1/2}e^{-inx}-\phi_{\lambda_n}\|^2<\infty$. Since for $|n|>N$, the algebraic multiplicity is equal to one, the functions $\phi_{\lambda_n}$ are all actual eigenfunctions given by (cf.\ \eqref{eq:3.10})
		\begin{equation} \phi_{\lambda_n}(x)=\frac{1}{\sqrt{2\pi}}e^{-i\lambda_nx}\left(1+i\int_0^xe^{i\lambda_n t}K(t)dt\right)\:.
		\end{equation}
		Moreover, we write $\lambda_n=n+\eta_n$, where $|\eta_n|\leq |\langle \psi_n,K\rangle|\leq \frac{1}{4\pi}$. For $|n|\geq N$, we estimate
		\begin{equation}
			\|\psi_n-\phi_{\lambda_n}\|^2\leq\frac{2}{\sqrt{2\pi}}\left(\|e^{-inx}-e^{-i\lambda_nx}\|^2+\left\|e^{-i\lambda_nx}\int_0^xe^{i\lambda_n t}K(t)dt\right\|^2\right)=\frac{2}{\sqrt{2\pi}}[(I)+(II)]\:.\label{eq:3.97}
		\end{equation}
		We continue with estimating the first term $(I)$:
		\begin{align}
			\|e^{-inx}-e^{-i\lambda_nx}\|^2&=\int_0^{2\pi}\left|e^{-inx}-e^{-i(n+\eta_n)x}\right|^2dx=\int_0^{2\pi}\left|1-e^{-i\eta_nx}\right|^2dx\label{eq:3.98}\\
			&\leq \frac{9}{4}\int_0^{2\pi}|\eta_n|^2x^2dx=6\pi^3|\eta_n|^2\leq 6\pi^3|\langle \psi_n,K\rangle|^2\:, \label{eq:3.99}
		\end{align}
		where we have used that for $|z|\leq 1/2$, one has the estimate $|1-e^z|\leq \frac{3}{2}|z|$. This applies to our case since $|-i\eta_nx|=|\eta_n|x\leq |\langle \psi_n,K\rangle|2\pi\leq 1/2$, for $x\in[0,2\pi]$ and because $|\langle \psi_n,K\rangle|\leq \frac{1}{4\pi}$. 
		Next, we focus on the second term $(II)$:
		\begin{align}
			(II)&=\int_0^{2\pi}\left|e^{-i\lambda_nx}\int_0^xe^{i\lambda_nt}K(t)dt\right|^2dx\leq C_K\int_0^{2\pi}\left|\int_0^xe^{i\lambda_nt}K(t)dt\right|^2dx\\&\leq 2C_K\left(\int_0^{2\pi}\left|\int_0^x(e^{i\lambda_nt}-e^{int})K(t)dt\right|^2dx+\int_0^{2\pi}\left|\int_0^xe^{int}K(t)dt\right|^2dx\right)\\&\leq 2C_K\left(\int_0^{2\pi}\left(\int_0^{2\pi}\left|(e^{i\lambda_nt}-e^{int})K(t)\right|dt\right)^2dx+\int_0^{2\pi}\left|\int_0^xe^{int}K(t)dt\right|^2dx\right) \label{eq:3.102}
			\\&\leq 2C_K\left(\int_0^{2\pi}\|K\|^2\|e^{int}-e^{i\lambda_nt}\|^2dx+\int_0^{2\pi}\left|\int_0^xe^{int}K(t)dt\right|^2dx\right) \label{eq:3.103}
			\\&\leq 2C_K\left(12\pi^4|\langle \psi_n,K\rangle|^2\|K\|^2+\int_0^{2\pi}\left|\int_0^xe^{int}K(t)dt\right|^2dx\right)\:. \label{eq:3.104}
		\end{align}
		For the first inequality of this estimate, we used that by Theorem \ref{thm:3.14}, we have $\sup_{n\in\mathbb{Z}}|\Im\lambda_n|\leq C$ and consequently there is some $C_K<\infty$ such that $|e^{-i\lambda_nx}|\leq C_K$ for all $n\in\mathbb{Z}$ and all $x\in[0,2\pi]$. The inequality $\eqref{eq:3.102}\leq\eqref{eq:3.103}$ follows from Cauchy-Schwarz and the inequality $\eqref{eq:3.103}\leq\eqref{eq:3.104}$ follows from $\|e^{int}-e^{i\lambda_nt}\|^2\leq 6\pi^3|\langle \psi_n,K\rangle|^2$, which can be argued for analogously to \eqref{eq:3.98} and \eqref{eq:3.99}. 
		
		Taking everything together, we continue to estimate \eqref{eq:3.97}
		\begin{align}
			&\|\psi_n-\phi_n\|^2\\\leq &\frac{2}{\sqrt{2\pi}}\left((6\pi^3+24C_K\pi^4\|K\|^2)|\langle \psi_n,K\rangle|^2+2C_K\int_0^{2\pi}\left|\int_0^xe^{int}K(t)dt\right|^2dx\right)    \\
			=&D_K|\langle \psi_n,K\rangle|^2+E_K\int_0^{2\pi}\left|\int_0^xe^{int}K(t)dt\right|^2dx\:,
		\end{align}
		for some constants $D_K,E_K<\infty$. We then get
		\begin{align}
			&\sum_{|n|>N}\|\psi_n-\phi_n\|^2\leq D_K\sum_{|n|>N}|\langle \psi_n,K\rangle|^2+E_K\sum_{|n|>N}\int_0^{2\pi}\left|\int_0^xe^{int}K(t)dt\right|^2dx\\\leq & D_K\sum_{n\in\mathbb{Z}}|\langle \psi_n,K\rangle|^2+E_K\sum_{n\in\mathbb{Z}}\int_0^{2\pi}\left|\int_0^xe^{int}K(t)dt\right|^2dx\label{eq:3.109}\\=&D_K\|K\|^2+\sqrt{2\pi}E_K\int_0^{2\pi}\left(\sum_{n\in\mathbb{Z}}|\langle \psi_n, \chi_{[0,x]}K\rangle|^2\right)dx\label{eq:3.110}\\= & D_K\|K\|^2+\widehat{E}_K\int_0^{2\pi}\|\chi_{[0,x]}K\|^2dx\leq D_K\|K\|^2+\widehat{E}_K\int_0^{2\pi}\|K\|^2dx\label{eq:3.111}\\=&(D_K+2\pi\widehat{E}_K)\|K\|^2<\infty\:,
		\end{align}
		where $\widehat{E}_K:=\sqrt{2\pi}E_K$. We used Fubini's theorem to interchange the summation with the integral in \eqref{eq:3.109} and $\chi_{[0,x]}$ in \eqref{eq:3.110} denotes the indicator function of the interval $[0,x]$. For the inequality in \eqref{eq:3.111}, we estimated $\|\chi_{[0,x]}K\|\leq\|K\|$ for all $x\in[0,2\pi]$. This finishes the proof.
	\end{proof}
	
	\begin{remark} $(i)$ Since for $\mu_n\in\sigma(P_{1,K}^*)$ and $|n|>N$, where $N$ is chosen large enough as in the proof of Theorem \ref{thm:3.20}, the eigenfunctions of $P_{1,K}^*$ are given by
		\begin{equation}
			\widetilde{\phi}_{\mu_n}(x)=\frac{1}{\sqrt{2\pi}}e^{-i\mu_n x}\:,
		\end{equation}
		where $|n-\mu_n|\leq |\langle \psi_n,K\rangle|\leq \tfrac{1}{4\pi}$, since $\mu_n\in\sigma(P_{1,K}^*)$ implies $\lambda_n=\overline{\mu_n}\in\sigma(P_{1,K})$ by the spectral mapping theorem. The estimate
		\begin{equation}
			\sum_{|n|>N}\|\widetilde{\phi}_n-\psi_n\|^2<\infty
		\end{equation}
		now follows analogously to \eqref{eq:3.98} and \eqref{eq:3.99}.
		\label{rem:3.21}
		
		$(ii)$ We also provide an alternative sketch of proof for the Riesz basis property of the root vectors of $P_{1,K}^*$. For $\mu_n\in\sigma(P_{1,K}^*)$, a direct calculation shows that the root spaces are given by
		\begin{equation} \label{eq:3.115}
			\mathcal{R}(\mu_n,P_{1,K}^*)=\mbox{span}\{e^{-i\mu_n x}, xe^{-i\mu_n x},\dots, x^m xe^{-i\mu_n x}\}\:,
		\end{equation}
		where $m=m_a(\mu_n,P_{1,K}^*)-1$. In particular, we know by Theorem \ref{thm:3.14}, that all but finitely many eigenvalues of $P_{1,K}^*$ have algebraic multiplicity equal to one and can be found inside disks around integers $n$, where $|n|>N$, with arbitrarily small radius. For such an eigenvalue, the root space coincides with the actual eigenspace and is given by
		\begin{equation}
			\mathcal{R}(\mu_n,P_{1,K}^*)=\ker(P_{1,K}^*-\mu_n)=\mbox{span}\{e^{-i\mu_nx}\}\:.
		\end{equation}
		On the other hand, it is a well-known result by Duffin/Eachus \cite{DE42, Y80} that if a sequence of complex numbers $\{\kappa_n\}_n$ satisfies $|\kappa_n-n|\leq L$ for some $L <\ln(2)/\pi$, then $\{e^{-i\kappa_n x}\}_{n\in\mathbb{Z}}$ is a Riesz basis of $L^2(-\pi,\pi)$ and using that $\sup_{n\in\mathbb{Z}}|\Im (\kappa_n)|\leq C<\infty$, it is also not difficult to show that it is also a Riesz basis of $L^2(0,2\pi)$. Choosing $N\in\mathbb{N}$ large enough such that $|\mu_n-n|\leq L<\ln(2)/\pi$ for all $|n|>N$, we define the sequence $\{\kappa_n\}_{n\in\mathbb{Z}}$, where
		\begin{equation}
			\kappa_n:=\begin{cases} \mu_n\quad&\mbox{if}\quad |n|>N\\ n\quad&\mbox{if}\quad |n|\leq N\:,\end{cases}
		\end{equation}
		which can immediately be seen to satisfy the assumptions of the theorem by Duffin/Eachus and hence, $\{e^{-inx}\}_{n=-N}^N\cup\{e^{-i\mu_nx}\}_{|n|>N}$ is a Riesz basis of $L^2(0,2\pi)$. It now follows from \cite[Lemma II.4.11]{AA95} that a change of finitely many exponents in such a Riesz basis yields another Riesz basis. A slight modification of the argument given in \cite[Lemma II.4.11]{AA95} also covers the case when finitely many exponents are changed to allow for repeated numbers, corresponding to root spaces of the form \eqref{eq:3.115}, cf.\ \cite[Section II.4.5]{AA95}. 
	\end{remark}

	\section{Dissipative extensions of $A_{min}$} \label{sec:4}
	For the remainder of this paper, we consider extensions $A_{\rho,k}$ of $A_{min}$ that are
	maximally dissipative. 
	
	\subsection{Basic definitions and characterization of all dissipative extensions of $A_{min}$}
	
	We recall the definition of (maximally) dissipative operators:
	\begin{definition}
		A densely defined operator $A$ in a Hilbert space $\mathcal{H}$ is called \emph{dissipative} if 
		\begin{equation}
			\Im\langle f,Af\rangle\geq 0
		\end{equation}
		for every $f\in\dom(A)$. Moreover, if there is no non-trivial dissipative operator extension of $A$, then $A$ is called \emph{maximally dissipative}. 
	\end{definition}
	There are the following equivalent characterizations of maximally dissipative operators:
	\begin{proposition}{\cite{P59}} Let $A$ be a dissipative operator in $\mathcal{H}$. Then the following are equivalent:
		\begin{enumerate}
			\item $A$ is maximally dissipative
			\item $-A^*$ is dissipative
			\item there is a $z\in\mathbb{C}^-$ with $z\in\varrho(A)$
			\item $\mathbb{C}^-\subseteq \varrho(A)$
			\item $(iA)$ generates a strongly continuous one-parameter semigroup of contractions
		\end{enumerate}
	\end{proposition}
	For $A_{min}$ to be dissipative, it is necessary and sufficient that $\Im V(x)=V_I(x)\geq 0$ almost everywhere, which we will assume from now on:
	
	\begin{hypothesis} For the remainder of the paper, we assume that $V_I(x)\geq 0$ almost everywhere and that $V_I(x)>0$ on a set of positive Lebesgue measure.
		\label{hypo:4.3}
	\end{hypothesis}
	
	Under this assumption, the operator $V_I$ is a bounded, non-negative, selfadjoint operator not equal to the zero operator and possesses a unique
	non-negative selfadjoint square-root, which we will denote by $V_I^{1/2}$. It is the maximal operator of multiplication by the function $V_I^{1/2}(x)$ and $\ran(V_I^{1/2})\subseteq L^2(0,2\pi)$ denotes the range of this operator, i.e.
	\begin{equation}
		\ran(V_I^{1/2})=\{V_I^{1/2}f\:|\: f\in L^2(0,2\pi)\}\:.
	\end{equation}
	
	From an application of the results obtained in \cite{F18}, we have the following characterization of \emph{all} maximally dissipative extensions of $A_{min}$:
	\begin{proposition}
		Let $\rho\in \mathbb{C}$ and $k\in L^2(0,2\pi)$. The operator $A_{\rho,k}$ given by
		\begin{align}
			A_{\rho,k}: \quad\dom(A_{\rho,k})&=\{f\in H^1(0,2\pi)\:|\: f(0)=\rho f(2\pi)\} \\(A_{\rho,k}f)(x)&=if'(x)+V(x)f(x)+f(2\pi)k(x)
		\end{align}
		is (maximally) dissipative if and only if $k\in\ran(V_I^{1/2})$ and 
		\begin{equation} \label{eq:3.6a}
			1-|\rho|^2\geq \frac{1}{2}\left\|V_I^{-1/2}k\right\|^2\:.
		\end{equation}
		
		Moreover, all maximally dissipative extensions of $A_{min}$ are of the above form, i.e.\ if $\hat{A}$ is a maximally dissipative extension of $A_{min}$, then there exist unique $\rho\in\C$ and $k\in \ran(V_I^{1/2})$ such that $\hat{A}=A_{\rho,k}$.
		\label{thm:4.2}
	\end{proposition}
	\begin{remark}
		The operator $V^{-1/2}_I$ in \eqref{eq:3.6a} denotes the inverse of $V_I^{1/2}$ on its range $\ran(V^{1/2}_I)$ which reduces the operator $V_I^{1/2}$. 
	\end{remark}

	\subsection{Extensions of $A_{min}$ with a real eigenvalue}
	
	In this section, we investigate how to construct maximally dissipative extensions $A_{\rho,k}$ that have real eigenvalues. 
	Similar work was done in \cite{FNW24}, where dissipative Schr\"odinger operators on the half-line were studied. In principle, it would be possible to apply the methods used \cite{FNW24} to obtain the results presented in this section. However, since the operators $A_{\rho,k}$ are first-order and have no essential spectrum, we are able to give more straightforward and shorter proofs and avoid having to use the more technical methods used in \cite{FNW24}.
	
	To begin with, we recall the following result:
	
	\begin{lemma}{(Sz.-Nagy)} Let $A$ be a maximally dissipative operator and let $\lambda\in\mathbb{R}$ be an eigenvalue of $A$ with corresponding eigenvector $f\in\dom(A)$, i.e.\ $Af=\lambda f$. Then, $f\in\dom(A^*)$ and $A^*f=\lambda f$.
		\label{lemma:5.1}
	\end{lemma}
	
	The first main result of this section is the following
	
	\begin{theorem} \label{thm:5.2}
		Let $A_{\rho,k}$ be a maximally dissipative extension of $A_{min}$. Then, we have the following:
		\begin{itemize}
			\item[(i)] If $\lambda\in\sigma(A_{\rho,k})\cap\mathbb{R}$, then $m_a(\lambda, A_{\rho,k})=1$. 
			\item[(ii)] $\sigma(A_{\rho,k})$ contains at most one real point. 
			\item[(iii)] For a given $\lambda\in\mathbb{R}$, there exists at most one extension $A_{\rho,k}$ such that $\lambda\in\sigma(A_{\rho,k})$.
			\item[(iv)] If $A_{\rho,k}$ has a real eigenvalue $\lambda$, then the corresponding eigenspace is orthogonal to the root spaces corresponding to all other eigenvalues of $A_{\rho,k}$.
		\end{itemize}
	\end{theorem}
	\begin{proof} Part $(i)$: Let $0\neq f\in\ker(A_{\rho,k}-\lambda)$, where $\lambda\in\mathbb{R}$. Assume that $m_a(\lambda,A_{\rho,k})\geq 2$. Then, there exists $g\in\mathcal{R}(\lambda,A_{\rho,k})$ be such that $(A_{\rho,k}-\lambda)g=f$. Using $(A_{\rho,k}-\lambda)f=(A_{\rho,k}^*-\lambda)f=0$, we then get
		\begin{equation}
			\|f\|^2=\langle f,f\rangle=\langle (A_{\rho,k}-\lambda)g,f\rangle=\langle g,(A_{\rho,k}^*-\lambda)f\rangle=0\:,
		\end{equation}
		which is a contradiction. Hence, we have $m_a(\lambda,A_{\rho,k})=1$.
		
		Part $(ii)$: Let $\lambda_1, \lambda_2\in\mathbb{R}$ with $\lambda_1\neq\lambda_2$. Suppose that $f_1, f_2\in\dom(A_{\rho,k})\cap\dom(A_{\rho,k}^*)$ are the eigenfunctions corresponding to $\lambda_1, \lambda_2$, respectively. It follows from a simple calculation that as the solutions to the eigenvalue equation $A_{\rho,k}^*f=if'+\overline{V}f=\lambda f$, the functions $f_1$ and $f_2$ are given by
		\begin{equation}
			f_j(x)=\exp\left(i\lambda_j(2\pi-x)-i\int_x^{2\pi}\overline{V(t)}dt\right)\quad\mbox{for}\quad j=1,2\:.
		\end{equation}
		Observe that $f_1(2\pi)=f_2(2\pi)=1$ and 
		\begin{equation} \label{eq:5.3}
			f_2(x)=e^{i(\lambda_2-\lambda_1)(2\pi-x)}f_1(x)\:.
		\end{equation}
		Now, since by Lemma \ref{lemma:5.1} we have
		\begin{equation}
			A_{\rho,k}f_j=if_j'+Vf_j+f_j(2\pi)k=if_j'+Vf_j+k=A_{\rho,k}^*f_j=if_j'+\overline{V}f_j \label{eq:5.4}
		\end{equation}
		for $j=1,2$, this implies
		\begin{equation}
			V_If_j=\frac{i}{2}k\quad\mbox{for}\quad j=1,2 \label{eq:5.5}
		\end{equation}
		and therefore $V_I(f_1-f_2)=0$. Using \eqref{eq:5.3}, this yields
		\begin{equation}
			V_I(x)(f_1(x)-f_2(x))=V_I(x)\left[1-e^{i(\lambda_1-\lambda_2)(2\pi-x)}\right]f_1(x)=0
		\end{equation}
		almost everywhere. Since $f_1(x)\neq 0$ for every $x\in(0,2\pi)$, $V_I(x)>0$ on a set of positive measure, and $\left[1-e^{i(\lambda_1-\lambda_2)(2\pi-x)}\right]=0$ for only finitely many points, this is a contradiction. Therefore $A_{\rho,k}$ has at most one real eigenvalue.
		
		Part $(iii)$: Let $\lambda\in\mathbb{R}$ and suppose there exist $(\rho_1,k_1)\neq(\rho_2,k_2)$ such that $\lambda\in\sigma(A_{\rho_1,k_1})\cap\sigma(A_{\rho_2,k_2})$. For $j=1,2$, let $f_j$ satisfy $A_{\rho_j,k_j}f_j=A_{\rho_j,k_j}^*f_j=\lambda f_j $. Since 
		\begin{equation}
			A_{\rho_j,k_j}^*f_j=if_j'+\overline{V}f_j=\lambda f_j\:,
		\end{equation}
		this means that $f_1$ and $f_2$ must satisfy the same first-order differential equation and therefore are equal (up to a multiplicative constant), which implies $\rho_1=\rho_2$. In what follows, let us choose $f_1$ and $f_2$ such that $f_1=f_2$. Arguing as in \eqref{eq:5.4} and \eqref{eq:5.5}, we get
		\begin{equation}
			\frac{i}{2}k_1=V_If_1=V_If_2=\frac{i}{2}k_2\:,
		\end{equation}
		and therefore $k_1=k_2$, which is a contradiction.
		
		Part $(iv)$: Let $\lambda$ be the real eigenvalue of $A_{\rho,k}$ and $g$ be a root vector corresponding to a non-real eigenvalue $\mu$. Then, there exists $j\in\mathbb{N}$ such that $(A_{\rho,k}-\mu)^jg=0$. Now, consider
		\begin{equation}
			(\lambda-\overline{\mu})^j\langle g,f\rangle=\langle g, (A_{\rho,k}-\overline{\mu})^jf\rangle=\langle g, (A^*_{\rho,k}-\overline{\mu})^jf\rangle=\langle (A_{\rho,k}-{\mu})^jg, f\rangle=0\:,
		\end{equation}
		which shows that $f\perp g$. This finishes the proof.
	\end{proof}
	
	Combining the results of Section \ref{sec:3} with the previous theorem, we have the following statement about $\sigma(A_{\rho,k})$:
	\begin{corollary} \label{thm:4.5}
		Let $A_{\rho,k}$ be a maximally dissipative extension of $A_{min}$, where $\rho\neq0$. We then have
		\begin{equation}
			\inf_{\mu\in\sigma(A_{\rho,k})\setminus\mathbb{R}}\Im \mu=:\zeta>0\:. \label{eq:4.15}
		\end{equation}
		Moreover, the root functions of $A_{\rho,k}$ form a Riesz basis of $L^2(0,2\pi)$.
	\end{corollary}
	\begin{proof}
		Firstly note that since $A_{\rho,k}$ is maximally dissipative, all of its eigenvalues have nonnegative imaginary part.
		From Hypothesis \ref{hypo:4.3} it follows that $\int_0^{2\pi}V_I(t)dt>0$. Moreover, since $A_{\rho,k}$ is assumed to be maximally dissipative, this implies by Proposition \ref{thm:4.2} that $|\rho|\leq 1$. Consequently, the scalar $\eta$, which is defined in \eqref{eq:3.7}, has positive imaginary part. Since by Lemma \ref{lemma:3.4}, $\sigma(A_{\rho,k})=\sigma(P_{1,K}+\eta)$, Theorem \ref{thm:3.14} implies that for any $\varepsilon>0$, all but finitely many eigenvalues of $A_{\rho,k}$ can be found inside circles with center points $(n+\eta)$ and radius $\varepsilon$, where $n\in\mathbb{Z}$ such that $|n|$ is large enough. Choosing $\varepsilon$ small enough, this ensures that the imaginary parts of all these eigenvalues will be larger than $\Im \eta/2$. Since there are only finitely remaining additional eigenvalues, this yields \eqref{eq:4.15}.
		The fact that the root functions of $A_{\rho,k}$ form a Riesz basis follows from Lemma \ref{lemma:3.4}. Since $W^{-1}(P_{1,K}+\eta)W=A_{\rho,k}$, the root functions of $A_{\rho,k}$ are given by $\{W^{-1}\phi_n\}_{n\in\mathbb{Z}}$, where $\{\phi_n\}_{n\in\mathbb{Z}}$ are the root functions of $P_{1,K}$, which by Theorem \ref{thm:3.20} form a Riesz basis. But this means that the root functions of $A_{\rho,k}$ are the image of the Riesz basis $\{\phi_n\}_{n\in\mathbb{Z}}$ under the boundedly invertible operator $W^{-1}$ and therefore also a Riesz basis.
	\end{proof}
	\begin{theorem} \label{thm:4.9}
		For any $\lambda\in\mathbb{R}$ there exists a unique pair $\rho_\lambda, k_\rho$ such that $A_{\rho_\lambda, k_\lambda}$ is a maximally dissipative extension of $A_{min}$ with $\lambda\in\sigma(A_{\rho_\lambda,k_\lambda})$.
	\end{theorem}
	\begin{proof} We prove this by an explicit construction. To this end, let $g_\lambda$ be the solution to the following initial value problem
		\begin{align}
			ig_\lambda'+\overline{V}g_\lambda&=\lambda g_\lambda \label{eq:5.8}\\ 
			g_\lambda(2\pi)&=-2i\:. 
		\end{align}
		We will show that the operator $A_{\rho_\lambda,k_\lambda}$ with the choice $\rho_\lambda:=\frac{i}{2}g_\lambda(0)$ and $k_\lambda:=V_Ig_\lambda$ is a maximally dissipative extension of $A_{min}$ with a real eigenvalue at $\lambda$ and $g_\lambda$ being the corresponding eigenvector. 
		
		Note that the solution to \eqref{eq:5.8} is given by 
		\begin{equation}
			g_\lambda(x)=\frac{2}{i}\exp\left(-i\int_x^{2\pi}\overline{V(t)}dt+(2\pi-x)i\lambda\right)
		\end{equation}
		and from a direct calculation, it can be verified that $1-|\rho_\lambda|^2=\tfrac{1}{2}\|V_I^{-1/2}k_\lambda\|^2$, which implies by Proposition \ref{thm:4.2}, Condition \eqref{eq:3.6a} that $A_{\rho_\lambda,k_\lambda}$ is maximally dissipative.
		
		Since $g_\lambda(0)=\frac{2}{i}\rho_\lambda=\rho_\lambda g_\lambda(2\pi)$, this implies that $g_\lambda\in\dom(A_{\rho_\lambda,k_\lambda})$. We now compute $A_{\rho_\lambda,k_\lambda}g_\lambda$:
		\begin{equation}
			A_{\rho_\lambda,k_\lambda}g_\lambda=ig_\lambda'+Vg_\lambda+g_\lambda(2\pi)k_\lambda=ig_\lambda'+(V_R+iV_I)g_\lambda-2iV_Ig_\lambda=ig_\lambda'+\overline{V}g_\lambda=\lambda g_\lambda\:,
		\end{equation}
		where the last equality follows from \eqref{eq:5.8}. This shows that $g_\lambda$ is an eigenfunction of $A_{\rho_\lambda,k_\lambda}$ corresponding to the eigenvalue $\lambda$. The uniqueness of $\rho_\lambda, k_\lambda$ follows from Theorem \ref{thm:5.2}, part $(iii)$.
	\end{proof}
	\subsection{Long-time behavior of the semigroup generated by $iA_{\rho,k}$}
	We finish this work with an investigation of the long-term behavior of the semigroup of contractions generated by $iA_{\rho,k}$. Depending on whether $\sigma(A_{\rho,k})$ contains a real eigenvalue or not, we observe two different types of behaviors:
	\begin{theorem} \label{thm:4.10}
		Let $\rho\neq0$ and $k\in\ran(V_I^{1/2})$ be such that $A_{\rho,k}$ is maximally dissipative and let $\left(e^{iA_{\rho,k}t}\right)_{t\geq 0}$ be the semigroup generated by $A_{\rho,k}$. Then, we have the following two cases:
		\begin{itemize}
			\item[(i)] If $\sigma(A_{\rho,k})\cap\mathbb{R}=\emptyset$, then $\lim_{t\rightarrow\infty}\left\|e^{iA_{\rho,k}t}\right\|=0$.
			\item[(ii)] If $\sigma(A_{\rho,k})\cap\mathbb{R}=\{\lambda\}$, then $\lim_{t\rightarrow\infty}\left\|e^{iA_{\rho,k}t}-e^{i\lambda t}P_{\ker(A_{\rho,k}-\lambda)}\right\|=0$.
		\end{itemize}
	\end{theorem}
	\begin{proof}
		Firstly, note that by Theorem \ref{thm:5.2}, part (iii), only the two cases described above are possible.
		
		Part $(i)$: If $\sigma(A_{\rho,k})\cap\mathbb{R}=\emptyset$, then it follows from Corollary \ref{thm:4.5} that $\inf_{\mu\in\sigma(A_{\rho,k})}\Im \mu=:\zeta>0$. Moreover, Corollary \ref{thm:4.5} also implies that the root functions $\{\phi_n\}_{n\in\mathbb{Z}}$ of $A_{\rho,k}$ form a Riesz basis of $L^2(0,2\pi)$. We may therefore expand any $f\in L^2(0,2\pi)$ in this basis: $f=\sum_{n\in\mathbb{Z}}c_n\phi_n$. For the time evolution, we then get
		\begin{align}
			\left\|e^{iA_{\rho,k}t}f\right\|^2=\left\|e^{iA_{\rho,k}t}\sum_{n\in\mathbb{Z}}c_n\phi_n\right\|^2=\left\|\sum_{n\in\mathbb{Z}}p_n(t)e^{i\mu_n t}c_n\phi_n\right\|^2\:, \label{eq:6.1}
		\end{align}
		where $\mu_n$ is the eigenvalue corresponding to the root function $\phi_n$ and $p_n(t)$ is a polynomial of degree less than or equal to $(m_a(\mu_n,A_{\rho,k})-1)$. In particular, if $m_a(\mu_n,A_{\rho,k})=1$, this implies $p_n(t)\equiv 1$. Since by Theorem \ref{thm:3.14} there are at most finitely many eigenvalues of $A_{\rho,k}$ with algebraic multiplicity greater than $1$, and therefore only finitely many $n\in\mathbb{Z}$, for which $p_n(t)$ is a non-constant polynomial, there is a large enough $T>0$ such that $\sup_{n\in\mathbb{Z}}|p_n(t)|\leq e^{\zeta t/2}$ for all $t\geq T$ and therefore
		\begin{align}
			|p_n(t)e^{i\mu_n t}|=|p_n(t)|e^{-\Im\mu_n t}\leq e^{\zeta t/2}e^{-\zeta t}=e^{-\zeta t/2} \label{eq:6.2}
		\end{align}
		for all $n\in\mathbb{Z}$ and all $t\geq T$. Since $\{\phi_n\}_{n\in\mathbb{Z}}$ is a Riesz basis, the operator $S: L^2(0,2\pi)\rightarrow L^2(0,2\pi)$, $\phi_n\mapsto \psi_n$ and then extended by linearity, where $\psi_n(x)=(2\pi)^{-1/2}e^{-inx}$, is boundedly invertible. Assuming $t\geq T$, we therefore continue
		\begin{align}
			&\eqref{eq:6.1}=\left\|\sum_{n\in\mathbb{Z}}p_n(t)e^{i\mu_n t}c_nS^{-1}\psi_n\right\|^2\leq \|S^{-1}\|^2\left\|\sum_{n\in\mathbb{Z}}p_n(t)e^{i\mu_n t}c_n\psi_n\right\|^2\\ &=\|S^{-1}\|^2\sum_{n\in\mathbb{Z}}|p_n(t)|^2e^{-2\Im\mu_n t}|c_n|^2\|\psi_n\|^2\overset{\eqref{eq:6.2}}{\leq}\|S^{-1}\|^2e^{-\zeta t}\sum_{n\in\mathbb{Z}}|c_n|^2\|\psi_n\|^2\\&= \|S^{-1}\|^2e^{-\zeta t}\left\|\sum_{n\in\mathbb{Z}}c_n\psi_n\right\|^2=\|S^{-1}\|^2e^{-\zeta t}\left\|\sum_{n\in\mathbb{Z}}c_nS\phi_n\right\|^2\\&\leq \|S^{-1}\|^2\|S\|^2e^{-\zeta t}\|f\|^2\:.
		\end{align}
		For $f\neq 0$ and $t\geq T$, this implies 
		\begin{equation}
			\frac{\left\|e^{iA_{\rho,k}t}f\right\|}{\|f\|}\leq \|S^{-1}\|\|S\|e^{-\zeta t/2}
		\end{equation}
		and therefore $\lim_{t\rightarrow \infty}\left\|e^{iA_{\rho,k}t}\right\|=0$ as claimed.
		
		Part $(ii)$: Now assume that there is a real number $\mu_0$ in $\sigma(A_{\rho,k})$ with $\phi_0$ being the corresponding eigenvector. By Theorem \ref{thm:5.2}, part $(i)$, the algebraic multiplicity of $\mu_0$ is equal to $1$ and therefore we have $e^{iA_{\rho,k}t}\phi_0=e^{i\mu_0t}\phi_0$. Moreover, by Theorem \ref{thm:5.2}, part $(iv)$, all other root vectors $\{\phi_n\}_{n\neq 0}$ are orthogonal to $\phi_0$ and thus $P_{\ker(A_{\rho,k}-\mu_0)}\phi_n=0$ for $n\neq 0$. We therefore get
		\begin{align}
			&\left\|(e^{iA_{\rho,k}t}-e^{i\mu_0 t}P_{\ker(A_{\rho,k}-\mu_0)})f\right\|^2\\=&\left\|\left(e^{iA_{\rho,k}t}-e^{i\mu_0t}\right)c_0\phi_0\right\|^2+\left\|e^{iA_{\rho,k}t}\sum_{n\neq 0} c_n\phi_n\right\|^2=\left\|e^{iA_{\rho,k}t}\sum_{n\neq 0} c_n\phi_n\right\|^2\:.
		\end{align}
		Since, again by Corollary \ref{thm:4.5}, we have $\inf_{n\neq 0}\Im\mu_n=:\zeta>0$, we now may mimic the argument given in Part $(i)$ of this proof and conclude that
		\begin{equation}
			\lim_{t\rightarrow\infty}\left\|e^{iA_{\rho,k}t}-e^{i\mu_0t}P_{\ker(A_{\rho,k}-\mu_0)}\right\|=0\:,
		\end{equation}
		which finishes the proof.
	\end{proof}
	\noindent {\bf Acknowledgments.}
	This paper is the result of an undergraduate summer research project supervised by the first author. We gratefully acknowledge the financial support of the Baylor math department which made this possible, especially Prof.\ Dorina Mitrea's encouragement and support. We appreciate Prof.\ Anton Baranov and Prof.\ Dmitry Yakubovich's feedback concerning their paper \cite{BY16}. We also thank Prof.\ Fritz Gesztesy for valuable feedback. 
	
	
\end{document}